\documentclass[final]{siamltex} 
\usepackage[latin1]{inputenc}
\usepackage[english,swedish,german]{babel}

\usepackage[all,light,draft]{draftcopy}
\usepackage{algorithmic}
\usepackage{algorithm}
\usepackage{sidecap}

\usepackage{graphicx}
\usepackage{psfrag}
\usepackage{amsmath}
\usepackage{tabularx}
\usepackage{verbatim}
\usepackage{float}
\usepackage{amssymb}
\usepackage{latexsym}
\usepackage{pifont}
\usepackage[loose]{subfigure}

\newcommand{\orth}{{\perp}}
\newcommand{\vect}{{\operatorname {vec}}}

\newcommand{\epsm}{\varepsilon_{\rm mach}}
\newcommand{\veps}{\varepsilon}

\newcommand{\CalD}{\mathcal{D}}

\newcommand{\RR}{\mathbb{R}}

\newcommand{\BBB}{\mathcal{B}}

\newcommand{\CC}{\mathbb{C}}

\newcommand{\MM}{\mathbb{M}}

\newcommand{\kmax}{{k_{\max}}}
\newcommand{\gramschmidt}{{\tt gram\_schmidt}}
\newcommand{\infarnoldiexp}{{\tt infarn\_exp}}

\newtheorem{remark}[theorem]{Remark}

\newtheorem{defn}[theorem]{Definition}

\begin{document}

\title{
Computing a partial Schur factorization of\\ nonlinear
eigenvalue problems using\\ the infinite Arnoldi method
%
} 
\author{Elias Jarlebring, Karl Meerbergen, Wim Michiels}

\renewcommand{\thefootnote}{\fnsymbol{footnote}}

%
\renewcommand{\thefootnote}{\arabic{footnote}}



\selectlanguage{english}
\maketitle
\begin{abstract}
%
The partial Schur factorization can be used to  represent 
several eigenpairs of a matrix in a numerically robust way.
Different adaptions of the Arnoldi
method are often used to compute partial Schur factorizations.
We propose here a technique to compute a partial Schur factorization of a nonlinear eigenvalue 
problem (NEP). The technique is inspired by the
algorithm in \cite{Jarlebring:2010:TRINFARNOLDI}, now called the \emph{infinite
  Arnoldi method}. The infinite Arnoldi method 
is a method designed for NEPs, and
can be interpreted as Arnoldi's method applied to 
a linear infinite-dimensional operator, whose reciprocal eigenvalues are the solutions to the 
NEP. 
As a first result we show that the invariant
pairs of the operator are equivalent to invariant pairs of the NEP.
We characterize the structure of the invariant pairs of the operator and 
show how one can carry out a modification of   the infinite Arnoldi method by 
respecting the structure. This also allows us to naturally add 
the feature known as \emph{locking}.
We  nest this algorithm with an outer iteration, where
the infinite Arnoldi method for a particular type of structured functions is
appropriately restarted.
The restarting exploits the structure and is inspired by the well-known
implicitly restarted Arnoldi method for standard eigenvalue problems.
The final algorithm 
is applied to examples from a benchmark collection, showing
that both processing time 
and memory consumption can be considerably 
reduced with the restarting technique.

\end{abstract}

\section{Introduction}
The  \emph{nonlinear eigenvalue problem} (NEP) will in this paper be used 
to refer to the  problem to find
$\lambda\in\Omega\subseteq\CC$ and $v\in\CC^n\backslash\{0\}$ such that 
\begin{equation}
  M(\lambda)v=0, \label{eq:nlevpM}
\end{equation}
where $M:\Omega\rightarrow\CC^{n\times n}$ is analytic in $\Omega$, which is
an open disc centered 
at the origin.

This problem class has received a considerable amount
of attention in the literature. 
See, e.g., the survey papers 
\cite{Mehrmann:2004:NLEVP,Ruhe:1973:NLEVP} and the monographs
\cite{Lancaster:2002:LAMBDA,Gohberg:1982:MATRIXPOLYNOMIALS}.
The results for \eqref{eq:nlevpM} are often (but not always)  presented with
some restriction of the 
generality of $M$, such as the theory and algorithms 
for polynomial eigenvalue problems (PEPs) in \cite{Lancaster:2002:LAMBDA,Lancaster:2005:PSEUDOSPECTRA,Mackey:2006:STRUCTURED,Fassbender:2008:PCP}, 
in particular 
the algorithms for quadratic eigenvalue problems (QEPs) \cite{Tisseur:2001:QUADRATIC,Bai:2005:SOAR,Meerbergen:2001:LOCKING}, 
but also 
recent approaches for 
rational eigenvalue problems (REPs) \cite{Su:2008:REP,Voss:2003:MAXMIN}.
The results we will now present are directly related to \cite{Jarlebring:2010:TRINFARNOLDI}
where an algorithm is presented which we here call the \emph{infinite Arnoldi method}.
An important aspect of the algorithm in this paper, and the infinite Arnoldi method,
is \emph{generality}.
Although the algorithm and results of this paper are applicable to PEPs, QEPs and REPs, 
the primary goal of the paper is not to solve problems
for the most common structures,
but rather 
to construct an algorithm which can be applied to other,
less common NEPs in a somewhat automatic fashion. 
Some less common NEPs are  given in the problem collection 
\cite{Betcke:2010:NLEVPCOLL}; there exists NEPs with
 exponential terms  \cite{Jarlebring:2010:DELAYARNOLDI}
and implicitly
stated NEPs such as \cite{Rott:2010:ITERATIVE}.   

In this paper we will present a procedure to
 compute a partial Schur factorization in the sense of the
concepts of partial Schur factorizations and invariant pairs for nonlinear eigenvalue
problems introduced in  
 \cite{Kressner:2009:BLOCKNEWTON}. 
These concepts  can be summarized as follows. First note that the function $M$ 
is in this work assumed to be analytic and  
can always be decomposed as a sum of 
products of constant matrices and scalar nonlinearities,  
\begin{equation}
   M(\lambda)=M_1f_1(\lambda)+\cdots+M_mf_m(\lambda),\label{eq:M1f1}
\end{equation}
where $f_i:\Omega\rightarrow\CC$, $i=1,\ldots,m$ are analytic in
$\Omega$. 
%
We define 
\[
  \MM(Y,\Lambda):=M_1Yf_1(\Lambda)+\cdots+M_mYf_m(\Lambda),
\]
where $f_i(\Lambda)$, $i=1,\ldots,m$ are the matrix functions corresponding to $f_i$, 
which are well defined if $\sigma(\Lambda)\subset\Omega$. 
An invariant pair $(Y,\Lambda)\in\CC^{n\times p}\times\CC^{p\times p}$ 
(in the sense of \cite[Definition~1]{Kressner:2009:BLOCKNEWTON}) satisfies 
\begin{equation}
\MM(Y,\Lambda)=0.\label{eq:invpairMM}
\end{equation}
Additional appropriate orthogonality conditions for $Y$ and $\Lambda$ 
yield a consistent definition of invariant pairs and 
the eigenvalues of $\Lambda$ are 
solutions to the nonlinear eigenvalue problem.
In this setting, a partial Schur factorization corresponds to  a particular invariant pair where $\Lambda$ 
is an upper triangular matrix. 
%


The results of this paper are based on a reformulation of the problem of
finding an invariant pair of the nonlinear eigenvalue problem
%
as a corresponding problem
formulated with a (linear) infinite dimensional operator denoted $\BBB$,
also used in the infinite Arnoldi method  \cite{Jarlebring:2010:TRINFARNOLDI}. 
In \cite{Jarlebring:2010:TRINFARNOLDI}, we presented an algorithm
which can be interpreted as Arnoldi's method applied to the operator 
$\BBB$. Although the operator $\BBB$ maps functions to functions, it turns out 
that the algorithm can be implemented with finite-dimensional linear algebra
operations if the Arnoldi method (for $\BBB$) is started with a constant function. This 
results in a Krylov subspace consisting of polynomials. 
Unlike the polynomial setting in \cite{Jarlebring:2010:TRINFARNOLDI}, we will
in this work consider linear combinations of 
exponentials and polynomials allowing us to carry out an efficient restarting process. We will show that similar to the polynomial setting \cite{Jarlebring:2010:TRINFARNOLDI}, the Arnoldi method for $\BBB$ applied to linear combinations
of polynomials and exponentials can be carried out with 
finite-dimensional linear algebra operations.



The reformulation with the operator $\BBB$ allows us to adapt a procedure 
based on the Arnoldi method designed for the computation of a 
partial Schur factorization for standard eigenvalue problems. 
We will use a construction inspired by
the \emph{implicitly restarted Arnoldi method} (IRAM) 
\cite{Morgan:1996:RESTARTING,Sorensen:1992:IMPLICIT,Lehoucq:2001:IRAM,Lehoucq:1996:DEFLATION}.
The construction is first
outlined adaption is outlined in Section~\ref{sect:linearprobs}
and consists of two steps respectively given
in Section~\ref{sect:lockedinfarn} and 
Section~\ref{sect:restart}. They correspond to carrying out the Arnoldi 
method for the operator $\BBB$ with a locked invariant pair and a 
procedure to restart it. 

%

We finally wish to mention that there exist restarting schemes
for algorithms for special cases of \eqref{eq:nlevpM}, in particular for QEPs \cite{Zhou:2010:RESTARTED,Jia:2010:RSOAR}.


%
\section{Reformulation as infinite-dimensional operator problem}\label{sect:reformulation}
In order to characterize the invariant pairs of \eqref{eq:nlevpM}
for our setting
we first need to introduce some notation. 
The function
$B:\Omega\rightarrow\CC^{n\times n}$, will be defined by
\begin{equation}
  B(\lambda):=M(0)^{-1}\frac{M(0)-M(\lambda)}{\lambda}\label{eq:Bdef}
\end{equation}
for $\lambda\in\Omega\backslash\{0\}$ and 
defined  as the analytic continuation at $\lambda=0$. Note
that $B$ is also analytic in $\Omega$, under the condition that $\lambda=0$ is not a
solution to \eqref{eq:nlevpM}. We will in this work 
assume that the NEP is such that  $\lambda=0$ is not 
an eigenvalue. From the  definition \eqref{eq:Bdef} we reach 
a transformed nonlinear eigenvalue problem 
\begin{equation}
  \lambda B(\lambda)x=x.\label{eq:nlevpB}
\end{equation}
%
We will also use a decomposition of $B$ similar to the decomposition \eqref{eq:M1f1} of $M$. That is, we let
\begin{equation}
  B(\lambda)=B_1b_1(\lambda)+\cdots+B_mb_m(\lambda),
\end{equation}
where $b_i:\Omega\rightarrow\CC$, $i=1,\ldots,m$ are analytic in
$\Omega$. Moreover, we will use the straightforward coupling of 
the decomposition of $M$, 
by setting
\begin{equation}
   B_i=M(0)^{-1}M_i,\;\;b_i(\lambda)=\frac{f_i(0)-f_i(\lambda)}{\lambda }.\label{eq:BiMi}
\end{equation}

We will use the following notation in order to express the operator
and carry out manipulations of the operator in a concise way. Let the differentiation operator $B(\frac{d}{d\theta})$
be defined by the Taylor expansion in a consistent way, i.e., 
\[
  \left(B(\frac{d}{d\theta})\varphi\right)(\theta):=
B(0)\varphi(\theta)+\frac{1}{1!}B'(0)\varphi'(\theta)+
\frac{1}{2!}B''(0)\varphi''(\theta)+\cdots,
\]
where $\varphi:\CC\rightarrow\CC^n$ is a smooth function.
We are now ready to introduce the operator which serves as the 
basis for the algorithm.
%
\begin{defn}[The operator $\BBB$]\label{def:B}
Let $\BBB$ denote the map defined 
by the domain  $\CalD(\BBB):=\{\varphi\in C_\infty(\RR,\CC^n):
\sum_{i=0}^\infty B^{(i)}(0)\varphi^{(i)}(0)/(i!) \textrm{ is finite}\}$ and 
the action
\begin{equation}
  (\BBB\varphi)(\theta)=\int_{0}^\theta\varphi(\hat{\theta})\,d\hat{\theta}+C(\varphi),\label{eq:defB}
\end{equation}
where 
\begin{equation}
  C(\varphi):=   \sum_{i=0}^\infty \frac{1}{i!}B^{(i)}(0)\varphi^{(i)}(0)=
  \left(B(\frac{d}{d\theta})\varphi\right)(0).\label{eq:defBC}
\end{equation}
\end{defn}

Several properties of the operator $\BBB$ are characterized in
\cite{Jarlebring:2010:TRINFARNOLDI}. Most importantly, its reciprocal
eigenvalues are the solutions to \eqref{eq:nlevpB} and hence also to \eqref{eq:nlevpM} if $\lambda\neq 0$.  In this work we will need a more general result, characterizing the invariant pairs of $\BBB$.

To this end  we first define the
application 
 of the operator $\BBB$ to block functions, and say that
if  $\Psi:\CC\rightarrow\CC^{n\times p}$ with columns given by
\[
  \Psi(\theta)=(\psi_1(\theta),\ldots,\psi_p(\theta)),
\]
then $\BBB \Psi$ is interpreted in a block fashion, i.e., 
\[
  (\BBB \Psi)(\theta):=(\BBB\psi_1(\theta),\ldots,\BBB\psi_p(\theta)).
\]
With this notation, we can now consistently define an invariant pair
as a pair $(\Psi,R)$ of the operator $\BBB$, 
where $\Psi:\CC\rightarrow\CC^{n\times p}$ and
$R\in\CC^{p\times p}$ such that 
\begin{equation}
   (\BBB \Psi)(\theta)=\Psi(\theta)R.\label{eq:invpairdef}
\end{equation}
%

The following theorem explicitly shows the structure of the function $\Psi$ and 
relates invariant pairs of the operator with invariant
pairs \eqref{eq:invpairMM}, i.e., invariant pairs
in the setting in \cite{Kressner:2009:BLOCKNEWTON}.

%
%

\begin{theorem}[Invariant pairs of $\BBB$]\label{thm:invpairs}
Suppose $\Lambda\in\CC^{p\times p}$ is invertible and suppose 
$(\Psi,\Lambda^{-1})$ is an invariant pair of $\BBB$. 
Then, $\Psi$ can be expressed as, 
\begin{equation}
  \Psi(\theta)=Y\exp(\theta \Lambda),\label{eq:Fstruct}
\end{equation}
for some matrix $Y\in\CC^{n\times p}$. 
Moreover, given $\Lambda\in\CC^{p\times p}$ and $Y\in\CC^{n\times p}$ where $\Lambda$ is invertible, the following statements are equivalent:
\begin{itemize}
\item[i)] The pair $(\Psi,\Lambda^{-1})$, where $\Psi(\theta):=Y\exp(\theta \Lambda)$, is an invariant pair of the operator
  $\BBB$, i.e., 
\[
   (\BBB \Psi)(\theta)=\Psi(\theta)\Lambda^{-1}.
\]
\item[ii)] The pair $(Y,\Lambda)$ is an invariant pair of the
  nonlinear eigenvalue problem \eqref{eq:nlevpM} in the sense of
  \cite[Definition~1]{Kressner:2009:BLOCKNEWTON}, i.e.,
\begin{equation}
  \MM(Y,\Lambda)=0.
\end{equation}
\end{itemize}
\end{theorem}
\begin{proof}
By differentiating (with respect to $\theta$) the left and right-hand side of the definition of
an invariant pair \eqref{eq:invpairdef} and using that the action of 
$\BBB$ is integration, we find that $\Psi$ satisfies the matrix differential equation,
\[
   \Psi(\theta)=\Psi'(\theta)\Lambda^{-1}.
\]
By multiplying by $\Lambda$ and vectorizing the equation, we have
\[
   (\Lambda^T\otimes I)\vect(\Psi(\theta))=\frac{d}{d\theta}\vect(\Psi(\theta)),
\]
and we can form an explicit solution,
\[
   \vect(\Psi(\theta))=\exp(\theta \Lambda^T\otimes I)y_0=(\exp(\theta \Lambda)^T\otimes I)y_0.
\]
The conclusion \eqref{eq:Fstruct}  follows by reversing the vectorization and setting $\vect(Y)=y_0$.

The equivalence between statements i) and ii) follows directly from the fact that $M(0)$ 
is invertible (since $\Lambda$ is invertible and $\lambda=0$ is not an eigenvalue) and
the application of Lemma~\ref{thm:structinvres}.
\end{proof}

\section{Outline of the algorithm}\label{sect:linearprobs}

We now know (from  Theorem~\ref{thm:invpairs}) that an invariant pair
of the NEP \eqref{eq:nlevpM} is  equivalent 
to an invariant pair of the linear operator $\BBB$.
The general idea of the procedure we will present in later sections is 
inspired by the procedures used to compute partial Schur factorizations for 
 standard eigenvalue problems with the Arnoldi method \cite{Lehoucq:1996:DEFLATION,Sorensen:1992:IMPLICIT,Stewart:2001:KRYLOVSCHUR}. 
We will carry out a variant of  the corresponding algorithm for 
the operator $\BBB$. More precisely, we will repeat the following two steps. 

In the first step (described in Section~\ref{sect:lockedinfarn}) we compute, in a particular way, an orthogonal projection 
of the operator $\BBB$ onto a Krylov subspace. 
The projection is 
constructed such that it possesses the feature known as \emph{locking}. This
here means that given a partial Schur factorization (or an approximation of the partial Schur factorization) the projection respects
the invariant subspace and returns an approximation containing
the invariant pair (without modification)
and also approximations of further eigenvalues.
This prevents repeated
convergence to the eigenvalues in the locked 
partial Schur factorization in a robust way.

%
In the literature (for standard eigenvalue problems) this projection 
is often 
computed
with a variation of the Arnoldi method. 
 More precisely, we will start the Arnoldi algorithm 
with a state containing the (locked) partial Schur factorization. 
The result of the infinite Arnoldi method 
can be expressed as 
what is commonly called an \emph{Arnoldi factorization},
\begin{equation}
     (\BBB F_k)(\theta)=F_{k+1}(\theta) 
\underline{H}_k,\label{eq:arnfact}
\end{equation}
where $\underline{H}_k\in\CC^{(k+1)\times k}$ is a Hessenberg matrix, 
$F_ {k+1}:\CC\rightarrow\CC^{n\times k}$ is an orthogonal 
basis of the Krylov subspace and $F_k$ is the first $k$ columns of $F_{k+1}$.
In this paper we  use a common notation for Hessenberg matrices; the first $k$ rows of the matrix $\underline{H}_k$ 
will be denoted $H_k\in\CC^{k\times k}$.
A property of the locking feature
is that 
the Hessenberg matrix $\underline{H}_k$ has  the structure
\[
\underline{H}_k=
\begin{pmatrix}
(\underline{H}_k)_{1,1} & (\underline{H}_k)_{1,2}\\
  & (\underline{H}_k)_{2,2}
\end{pmatrix}
\]
where $R=(\underline{H}_k)_{1,1}\in\CC^{p_l\times p_l}$ is an upper triangular matrix. The upper left block of the $\underline{H}_k$ is called the \emph{locked} part, 
since the first $p_l$ columns of \eqref{eq:arnfact} is the equation for an  invariant pair \eqref{eq:invpairdef}.

In the first  step we show how the Arnoldi method with locking
can be carried out if we represent the functions in the algorithm (and 
in the factorization \eqref{eq:arnfact}) in a structured way. Unlike
the infinite Arnoldi method in  \cite{Jarlebring:2010:TRINFARNOLDI} 
we will need to work with functions 
which are linear combinations of 
exponentials and polynomials. 
It turns out that, similar to \cite{Jarlebring:2010:TRINFARNOLDI},
the action of the operator as well as the entire Arnoldi algorithm
can be carried out with finite-dimensional arithmetic, while the use
of exponentials is benificial also for the second step.

In the second step (described in Section~\ref{sect:restart}), i.e., after computing the Arnoldi factorization \eqref{eq:arnfact}, we process  
the factorization such that two 
types of information can be  extracted.
\begin{itemize}
\item We extract converged eigenvalues from the Arnoldi factorization \eqref{eq:arnfact} and store those in a partial Schur factorization. 
Due to the locking feature, the updated partial Schur factorization will be 
of the same size as the locked part of \eqref{eq:arnfact}  or larger.  
\item We extract a function with  favorable approximation
properties for those eigenvalues of interest, which have not yet converged.
\end{itemize}
This information is extracted in a fashion
similar to 
implicitly restarted Arnoldi (IRAM) 
 \cite{Morgan:1996:RESTARTING,Lehoucq:2001:IRAM,Lehoucq:1996:DEFLATION,Stewart:2001:KRYLOVSCHUR}. However, several modifications
are necessary in order to restart with the structured functions.


The two steps are subsequently 
iterated by starting the (locked version) of Arnoldi's method with 
the extracted function and with (the possibly larger) partial Schur 
factorization. 
Thus, nesting the infinite Arnoldi method with a restarting scheme
which is expected to eventually converge to a partial Schur factorization.

\section{The infinite Arnoldi method with locked invariant pair}\label{sect:lockedinfarn}
In the first step of the conceptual algorithm described in 
Section~\ref{sect:linearprobs}, we need to carry out an
Arnoldi algorithm for $\BBB$ with the preservation feature that the
given partial Schur factorization is not modified.  
In an infinite-dimensional setting, the adaption to achieve this feature with the Arnoldi method 
is straightforward by initiating the state of the Arnoldi method with
the invariant pair. The procedure is given in
Algorithm~\ref{alg:arnoldiBBB}, where the basis of the invariant
subspace associated with the 
partial Schur factorization is assumed to be orthogonal with respect  to a given 
scalar product $<\cdot,\cdot>$.

%
%
%
%
%
%

\begin{algorithm}[h]
\caption{}
\label{alg:arnoldiBBB}
\begin{algorithmic}[1]
\INPUT 
A partial Schur factorization of $\BBB$  represented by $(\Psi,R)$ and a
function $f:\CC\rightarrow\CC^{n}$  
such that $<f,f>=1$ and such that $f$ is orthogonal to the columns of
$\Psi$ with respect to  $<\cdot,\cdot>$.
%
\OUTPUT An Arnoldi factorization of $\BBB$ represented by $(\varphi_1,\ldots,\varphi_\kmax)$ and $H_{\kmax+1,\kmax}$
\STATE Set $H_{p_l,p_p}=R$
\STATE Set $(\varphi_1,\ldots,\varphi_{p_l})=\Psi$
\STATE Set $\varphi_{p_l+1}=f$
\FOR {$k=p_l+1,\ldots,k_{\max}$ }
\STATE  $\psi= \BBB\varphi_{k}$
\FOR {$i=1,\ldots,k$}
\STATE  $h_{i,k}=<\psi,\varphi_i>$
\STATE  $\psi= \psi- h_{i,k}\varphi_i$
\ENDFOR
\STATE  $h_{k+1,k}=\sqrt{<\psi,\psi>}$
\STATE  $\varphi_{k+1}=\psi/h_{k+1,k}$
\ENDFOR 
\end{algorithmic}
\end{algorithm}

\subsection{Representation of structured functions}\label{sect:struct}
In later sections we will provide a specialization  
of all the steps in the abstract algorithm above (Algorithm~\ref{alg:arnoldiBBB})
such that we can implement it in 
finite-dimensional arithmetic.
The first step in the  conversion of Algorithm~\ref{alg:arnoldiBBB} into
a finite-dimensional algorithm is to select an appropriate 
starting function and an appropriate 
finite-dimensional representation of the functions.
%
%

%

In this work, we will consider functions 
which are sums of exponentials and polynomials with the structure
\begin{equation}
  \varphi(\theta)=Ye^{S\theta}c+q(\theta)\label{eq:exppolystruct}
\end{equation}
where $Y\in\CC^{n\times p}$, $S\in\CC^{p\times p}$, $c\in\CC^{p}$ and $q:\CC\rightarrow\CC^n$ is a vector of polynomials. 
Moreover, we let $S$ be a  block triangular matrix
\begin{equation}
   S=
\begin{pmatrix}
   S_{11} & S_{12}\\  
    0 & S_{22}
\end{pmatrix},\label{eq:Sstruct0}
\end{equation}
and set $S_{11}=R^{-1}\in\CC^{p_l\times p_l}$ where $p_l\le p$, where $R$ will later be chosen such that it is an approximation of the matrix in the Schur factorization.  This structure has a number of favorable properties important for our situation.
\begin{itemize}
\item The action of $\BBB$ applied to  functions of the type \eqref{eq:exppolystruct} can be carried out 
in an efficient way using only  finite-dimensional operations. 
This stems from the property that the action of $\BBB$ corresponds 
to integration and the set of polynomials and exponentials under consideration
are closed under integration.
Algorithmic details will be given in Section~\ref{sect:action}.
\item This particular structure allows 
the storing and orthogonalization 
against an invariant subspace, which, according to Theorem~\ref{thm:invpairs}, 
has exponential structure.
\item The structure provides a freedom
to choose the blocks $S_{12}$ and $S_{22}$. This 
allows us to appropriately restart the algorithm. 
Due to the
exponential structure illustrated in Theorem~\ref{thm:invpairs},
it will turn out to be natural to impose an exponential
structure on the Ritz functions in order to construct a
function $f$ to be used in the restart.
The precise choice of $S_{12}$, $S_{22}$ and $Y$ will 
be further explained in Section~\ref{sect:restart}.
\end{itemize}

In practice we also need to store 
the structured functions 
in some fashion, preferably with matrices and vectors. It is tempting to store
the exponential part and polynomial part of \eqref{eq:exppolystruct} 
separately, i.e., to 
store the exponential part 
 with the variable $Y$, $S$ and $c$ and the polynomial
part by coefficients in some polynomial basis, e.g., the coefficients
 $y_0,\ldots,y_{N-1}$  in the monomial basis
$q(\theta)=y_0+y_1\theta+\cdots+y_{N-1}\theta^{N-1}$. 
Although such an approach is natural from a theoretical perspective,
it is not adequate from a numerical perspective. This can be seen as follows.
Note that the Taylor expansion of the structured function \eqref{eq:exppolystruct} is
\begin{multline}
\varphi(\theta)=
(Yc+y_0)   +
\left(\frac{1}{1!}YSc+y_1\right)\theta +
\cdots +            
\left(\frac{1}{(N-1)!}YS^{N-1}c+y_{N-1}\right)\theta^{N-1}+\\
\left(\frac{1}{N!}YS^{N}c\right) \theta^{N}+
\left(\frac{1}{(N+1)!}YS^{N+1}c\right) \theta^{N+1}+\cdots.\label{eq:naive}
\end{multline}
A potential source of cancellation is apparent 
for the first
$N$ terms in \eqref{eq:naive}  if the polynomial
$q(\theta)$ approximates  $-Y\exp(\theta S)c$. 
This turns out to be a situation appearing in practice in this
algorithm, 
making the
storing of the structured functions in this separated form inadequate.

We will instead use a function representation where 
the coefficients in the Taylor expansion are not formed by sums. 
This can be achieved by  replacing the first $N$ terms in \eqref{eq:naive}
by new coefficients $x_0,\ldots,x_{N-1}$, i.e.,
\begin{equation}
  \varphi(\theta)
=\\
x_0+x_1\theta+\cdots+ x_{N-1}\theta^{N-1}+
\frac{1}{N!}(YS^{N}c)\theta^N+
\frac{1}{(N+1)!}(YS^{N+1}c)\theta^{N+1}+\cdots.\label{eq:phistruct}
\end{equation}
The structured functions \eqref{eq:exppolystruct}
will be represented with the four variables $Y\in\CC^{n\times p}$,
$S\in\CC^{p\times p}$, $c\in\CC^p$,
$x\in\CC^{Nn}$, where $x^T=(x_0^T,\ldots,x_{N-1}^T)$.
Note that this representation does not suffer from the potential 
cancellation effects present in the naive representation \eqref{eq:naive}.

%

Throughout this work we will need to carry out many manipulations of functions 
represented 
in the form \eqref{eq:phistruct} and we
need a  concise notation.
Let $\exp_N$ denote the
remainder part of the truncated Taylor expansion of the exponential, 
i.e., 
\begin{equation}
  \exp_N(\theta S):=\exp(\theta S)-I-\frac{1}{1!}S-\cdots-\frac{1}{N!}S^N.\label{eq:defexpN}
\end{equation}
This can equivalently be expressed as, 
\begin{equation}
  \exp_N(\theta S)=
\frac{1}{(N+1)!}\theta^{N+1}S^{N+1}+
\frac{1}{(N+2)!}\theta^{N+2}S^{N+2}+\cdots\label{eq:defexpN2}
\end{equation}
with
\[
\exp_{-1}(\theta S):=\exp(\theta S).
\]
With this notation, we can now concisely express \eqref{eq:phistruct} with
$\exp_N$ and Kronecker products, 
\begin{equation}
  \varphi(\theta)=   
Y \exp_{N-1}(\theta S)c+
\left((1,\theta,\theta^2,\ldots,\theta^{N-1})\otimes I_n\right)x.\label{eq:phistruct2}
\end{equation}

\subsection{Action for structured functions}\label{sect:action}

We have now (in Section~\ref{sect:struct}) introduced 
the function structure and shown how we can
 represent these functions
with matrices and vectors. 
An important component in Algorithm~\ref{alg:arnoldiBBB} is the action of $\BBB$.  We will now show how we can compute the action of $\BBB$ applied to
 a  function given with the representation 
\eqref{eq:phistruct2}.

Analogous to the definition of $\exp_N$, it will be convenient to 
introduce a notation for the remainder part of the nonlinear eigenvalue problem $\MM$ after
a Taylor expansion to order $N$. We define, 
\begin{equation}
 \MM_N(Y,S):=\MM(Y,S)-M(0)Y-\frac{1}{1!}M'(0)YS-\frac{1}{2!}M''(0)YS^2-\cdots-
\frac{1}{N!}M^{(N)}(0)YS^{N}\label{eq:MMNdef} 
\end{equation}
or equivalently,
\begin{equation}
  \MM_N(Y,S)=
\frac{1}{(N+1)!}M^{(N+1)}(0)YS^{N+1}+
\frac{1}{(N+2)!}M^{(N+2)}(0)YS^{N+2}+\cdots.\label{eq:MMNdef2} 
\end{equation}
Note that with this definition
\begin{eqnarray*}
   \MM_{-1}(Y,S)&=&\MM(Y,S).
\end{eqnarray*}

We are now ready to express the  action of $\BBB$ applied 
to functions with the structure \eqref{eq:phistruct2}.
Note that  the construction of the new function $\varphi_+=\BBB \varphi$ in the following result only
involves standard linear algebra operations of matrices and vectors.

\begin{theorem}[Action for structured functions]\label{thm:structaction}
Let $S\in\CC^{p\times p}$ and
$c\in\CC^p$ be given constants, where $S$ is invertible.
Suppose 
\begin{equation}
  \varphi(\theta)= 
Y\exp_{N-1}(\theta S)c+
\left((1,\theta,\theta^2,\ldots,\theta^{N-1})\otimes I_n\right)x\label{eq:structaction:phi}
\end{equation}
Then, 
\begin{equation}
\varphi_+(\theta):=(\BBB\varphi)(\theta)= 
Y\exp_{N}(\theta S)c_++
\left((1,\theta,\theta^2,\ldots,\theta^{N})\otimes I_n\right)x_+\label{eq:phi+}
%
%
%
\end{equation}
where
\begin{equation}
 c_+=S^{-1}c,\label{eq:c+}
\end{equation}
\begin{equation}
  (x_{+,1},\ldots,x_{+,N}) = (x_0,\ldots,x_{N-1})
\begin{pmatrix}
1 &                     &&\\ 
  & \frac12             &&\\
  &                     &\ddots&\\
  &                     & &\frac1{N}\\
\end{pmatrix},\label{eq:X+}
\end{equation}
and 
\begin{equation}
x_{+,0}=
-M(0)^{-1}
\left(
\MM_{N}(Y,S)c_+
+
\sum_{i=1}^{N}M^{(i)}(0)x_{+,i}\right).
\label{eq:y0}
\end{equation}
\end{theorem}
\begin{proof}
We show that $\varphi_+$ constructed by \eqref{eq:c+}, \eqref{eq:X+} and \eqref{eq:y0} satisfies
\begin{equation}
  \BBB\varphi=\varphi_+,\label{eq:y0proof1}
\end{equation}
by first showing that the derivative of the 
left and the derivative of the right-hand side of \eqref{eq:y0proof1} are equal and 
then showing that they are also equal in one point $\theta=0$.
From the property \eqref{eq:defexpN2}, we have
\begin{equation}
  \frac{d}{d\theta}\exp_N(\theta S)=
\frac{1}{N!}\theta^{N}S^{N+1}+
\frac{1}{(N+1)!}\theta^{N+1}S^{N+2}+\cdots
=\exp_{N-1}(\theta S)S.\label{eq:y0proofdexp}
\end{equation}
Moreover, the relation \eqref{eq:X+} implies that
\begin{equation}
\frac{d}{d\theta}\left((1,\theta,\theta^2,\ldots,\theta^{N})\otimes I_n\right)x_+
=\left((1,\theta,\theta^2,\ldots,\theta^{N-1})\otimes I_n\right)x.\label{eq:y0proofdpoly}
\end{equation}
Note that $\BBB$ corresponds to integration and the left-hand side of
\eqref{eq:y0proof1} is $\varphi$. The right-hand side can be 
differentiated using \eqref{eq:y0proofdexp} and \eqref{eq:y0proofdpoly}. 
We reach that the right-hand side of \eqref{eq:y0proof1} is $\varphi$ 
by using \eqref{eq:c+}.

%
We have shown that the derivative of the 
left and the derivative of the right-hand side of \eqref{eq:y0proof1} are equal. 

We now evaluate \eqref{eq:y0proof1} at  $\theta=0$. From the definition of $\BBB$ 
we have that
that 
$(\BBB\varphi)(0)=\left(B(\frac{d}{d\theta})\varphi\right)(0)$, i.e., we wish to
show that
\begin{equation}
  (\BBB \varphi)(0)=\left(B(\frac{d}{d\theta})\varphi\right)(0)=
 \varphi_+(0)=x_0,\label{eq:y0proof2} 
\end{equation}
when $N>0$. (The relation obviously holds for $N=0$.) 
Note that by construction $\varphi_+$ is a primitive function of $\varphi$.
From the relations between $f_i$, $b_i$, $M_i$, $B_i$, in \eqref{eq:BiMi} 
it follows that \eqref{eq:y0proof2} is equivalent to 
\begin{equation}
  0=\left(\left(M_1f_1(\frac{d}{d\theta})+\cdots+M_mf_m(\frac{d}{d\theta})\right)\varphi_+\right)(0)=
(M(\frac{d}{d\theta})\varphi_+)(0).\label{eq:y0proofphi+}
\end{equation}
We now consider the terms of $\varphi_+$ in \eqref{eq:phi+} separately. 
Note that for any analytic function $g:\Omega \rightarrow\CC$, 
we have
\[
  \left(g(\frac{d}{d\theta})\exp_N(\theta S)\right)(0)=g_N(S),
\]
where $g_N$ is the remainder term in the truncated Taylor expansion, 
analogous to $\exp_N$. It follows that, 
\begin{equation}
\left(\left(M_1f_1(\frac{d}{d\theta})+\cdots+M_mf_m(\frac{d}{d\theta})\right)
Y \exp_N(\theta S)c_+\right)(0)=\MM_N(Y,S)c_+.\label{eq:y0proofexp}
\end{equation}
For the polynomial part of $\varphi_+$ we have
\begin{equation}
(M(\frac{d}{d\theta})
\left((1,\theta,\theta^2,\ldots,\theta^{N})\otimes I_n\right)x_+)(0)=
\sum_{i=0}^NM^{(i)}(0)x_{+,i}.\label{eq:y0proofpoly}
\end{equation}
Note that $(M(\frac{d}{d\theta})\varphi_+)(0)$ is the sum of \eqref{eq:y0proofexp}. Hence, we have shown \eqref{eq:y0proofphi+} (and hence also \eqref{eq:y0proof2})
by using \eqref{eq:y0proofexp}, \eqref{eq:y0proofpoly} and 
the definition of $x_{+,0}$  in \eqref{eq:y0}.
%
%
%

%

\end{proof}

\subsection{Scalar product and finite-dimensional specialization of Algorithm~\ref{alg:arnoldiBBB}}\label{sect:algorithmic}

Since the goal is 
to completely specify all operations in  Algorithm~\ref{alg:arnoldiBBB}
in a  finite-dimensional setting, 
we also need to  provide a scalar product. 
In \cite{Jarlebring:2010:TRINFARNOLDI} we worked
with polynomials and we defined the scalar product
via the Euclidean scalar product on monomial or Chebyshev
coefficients. 
The structured functions described in Section~\ref{sect:struct} are
not polynomials.
We can however still define the scalar product
consistent with \cite{Jarlebring:2010:TRINFARNOLDI}. 
In this work we restrict the presentation to the 
consistent extension of the definition of the scalar products via 
the monomial coefficients. Given two functions
\[
  \varphi(\theta)=\sum_{j=0}^\infty \theta^jx_j,\;\;
  \psi(\theta)=\sum_{j=0}^\infty \theta^jz_j,\;\;
\] 
we define
\begin{equation}
 <\varphi,\psi>:=\sum_{i=0}^\infty z_i^Hx_i.\label{eq:scalarprod_def0}
\end{equation}
It is straightforward to show that \eqref{eq:scalarprod_def0} satisfies
the properties of a scalar product and 
 that the sum in  \eqref{eq:scalarprod_def0}
is always finite for functions of the considered structure.
The computational details for the scalar product and the 
orthogonalization process are postponed until the next section
(Section~\ref{sect:scalarprod}).

The combination of the above results, i.e., the choice of the representation 
of the function structure (Section~\ref{sect:struct}), the operator action
(Section~\ref{sect:action}) and the scalar product \eqref{eq:scalarprod_def0}, 
forms a  complete specialization of all the operations in
 Algorithm~\ref{alg:arnoldiBBB}. 
For reasons of numerical efficiency, 
we will 
slightly modify the
direct implementation of the operations.


%
%
%
%
%

Instead of representing the individual functions $\varphi_1,\ldots,\varphi_k$  
of the basis $(\varphi_1,\ldots,\varphi_k)$ 
we will use a block representation and denote 
\[
  F_k(\theta)=(\varphi_1,\ldots,\varphi_k).
\]
Now note that variables $Y$ and $S$ in the function structure 
\eqref{eq:phistruct2} are not modified in Theorem~\ref{thm:structaction} and
obviously not modified when forming linear combinations. Hence,
the variables $Y$ and $S$ can be kept constant throughout the algorithm. This allows
us to also use the structured representation \eqref{eq:phistruct2} directly
for
the block function $F_k$ instead 
of individually for $\varphi_1,\ldots,\varphi_k$.
In every point in the algorithm, there exist matrices $C_k$ and $V_k$  such that
\begin{equation}
   F_k(\theta)=Y\exp_{N-1}(\theta S)C_k+ ((1,\theta,\ldots,\theta^{N-1})\otimes I)V_k,\label{eq:Fk}
\end{equation}
with an appropriate choice of $N$.

%
%

The variable $N$ defining the length of the polynomial part of the structure 
needs to be adapted during the iteration. 
This stems from the fact that functions  $\varphi_+$ and $\varphi$ in  
Theorem~\ref{thm:structaction} are represented with 
polynomial parts of different length ($N-1$ and $N$). 
Hence, we need to increase $N$ by one 
after each application of $\BBB$. 
Fortunately, the corresponding 
increase of $N$ can be easily achieved 
by treating the leading element of exponential part 
as an element of the polynomial part. Here, this means using
the fact that
\begin{multline}
   F_k(\theta)=Y\exp_{N-1}(\theta S)C_k+ ((1,\theta,\ldots,\theta^{N-1})\otimes I)V_k=\\
Y\exp_{N}(\theta S)C_k +((1,\theta,\ldots,\theta^{N})\otimes I)
\begin{pmatrix}
V_k\\ \frac{YS^NC_k}{N!}
\end{pmatrix}.
\end{multline}
In this work, the starting function $f$ will be an exponential function, and
after the first application of $\BBB$, we need to 
expand the polynomial part with one block consisting of 
$\frac{YS^NC_k}{N!}$ with $N=0$. 
Since $k=p_l+1$ at the first application of $\BBB$, 
for an iteration corresponding to a given $k$,
we  need to expand the 
polynomial part of $F_k$ with one block row consisting of $\frac{YS^NC_k}{N!}$ 
with $N=k-p_l-1$. 
%

\begin{algorithm}[H]
\caption[]{
Infinite Arnoldi method with structured functions and locked pair\\
\phantom{\bf Algorithm X:}$[V_k,C,H_k]=$\infarnoldiexp$(c,S,Y,p_l,\kmax)$
}
\label{alg:infarnoldi}
\begin{algorithmic}[1]
\INPUT  Number of iterations $\kmax$, coefficients 
$Y\in\CC^{n\times p}$, $S\in\CC^{p\times p}$, $c\in\CC^{p}$, 
representing the normalized function $f$ given by \eqref{eq:starting_f} and
the locked part of the factorization corresponding to the invariant 
pair $(\Psi,R)$ with $\Psi$ given by \eqref{eq:Psi_locked} and  $R\in\CC^{p_l\times p_l}$ is given from the structure of $S$ in \eqref{eq:Sstruct}.
The functions corresponding to the columns of $\Psi$ as well as the function $f$ must be orthogonal.
\OUTPUT $V_{\kmax+1}\in\CC^{(\kmax+1)n\times (\kmax+1)}$, $C_{\kmax+1}\in\CC^{p\times (\kmax+1)}$, $\underline{H}_{\kmax}\in\CC^{(\kmax+1)\times \kmax}$
representing the factorization \eqref{eq:arnfact3}
\vspace{0.2cm}
\STATE Set $H_{p_l,p_l}=R$
\STATE Set  $C_{p_l+1}=\begin{pmatrix}e_1 & \ldots &e_{p_l}& c\end{pmatrix}$
\STATE Set $V_{p_l+1}=$empty matrix of size $0\times (p_l+1)$
\FOR {$k=p_l+1,\ldots,\kmax$ }
\STATE Compute $c_+$ according to \eqref{eq:c+} where $c=c_k$, i.e.,
$k$th column of $C_k$.
\STATE Let $x\in\CC^{(k-p_l-1)n}$  be the $k$th column of $V_k$
\STATE Compute $x_{+,1},\ldots,x_{+,k-p_l-1}\in\CC^{n}$ according to \eqref{eq:X+} 
\STATE Compute $x_{+,0}$ according to \eqref{eq:y0} with $N=k-p_l-1$.
\STATE Expand $V_k$ with one block row: 
\[
\underline{V}_k=\begin{pmatrix}V_k\\\frac{YS^{k-p_l-1}C_k}{(k-p_l-1)!}\end{pmatrix}
\]
\STATE $[c_\orth,x_\orth,h_k,\beta]=$\gramschmidt$(c_+,x_+,C_k,\underline{V}_k)$
%
\STATE Let $\underline{H}_k = \left[\begin{array}{cc}\underline{H}_{k-1} & h_k \\ 0 & \beta \end{array}\right] \in\CC^{(k+1)\times k}$
\STATE Expand $C_k$ by setting  $C_{k+1}=(C_k,c_{\orth})$
\STATE Expand $V_k$ by setting  $V_{k+1} = (\underline{V}_k,x_{\perp})$
\ENDFOR 
\end{algorithmic}
\end{algorithm}

With the block structure representation \eqref{eq:Fk} we
can now specialize
Algorithm~\ref{alg:arnoldiBBB} for the structured functions. 
The finite-dimensional implementation of Algorithm~\ref{alg:arnoldiBBB} is
given in 
Algorithm~\ref{alg:infarnoldi} and
 visually illustrated in Figure~\ref{fig:infarnoldi}. 

The input and output of the algorithm should be interpreted as follows. 
The variables $Y$, $S$, $c$ specify the starting function $f$ as well as
the locked part of the factorization. The starting function is given by 
\begin{equation}
f(\theta)=
Y\exp(\theta S)c
\label{eq:starting_f}
\end{equation}
and the locked part of the factorization (in Algorithm~\ref{alg:arnoldiBBB} 
denoted $(\Psi,R)$) corresponds to 
\begin{equation}
  \Psi(\theta)=Y\exp(\theta S)\begin{pmatrix}I_{p_l}\\0\end{pmatrix}\label{eq:Psi_locked}
\end{equation}
where $R\in\CC^{p_l\times p_l}$ is defined as the inverse of the leading block of $S$. Recall 
that $S$ is assumed to have the block triangular structure \eqref{eq:Sstruct0},
i.e., 
\begin{equation}
   S=\begin{pmatrix}R^{-1}& S_{12}\\0 & S_{22} \end{pmatrix},\label{eq:Sstruct}
\end{equation}
The output is a finite-dimensional representation of the factorization 
\begin{equation}
     (\BBB F_{\kmax} )(\theta)=F_{\kmax+1}(\theta) 
\underline{H}_{\kmax},\label{eq:arnfact3}
\end{equation}
where the block function $F_{\kmax+1}$ is given 
\begin{equation}
 F_{\kmax+1}(\theta)=
Y\exp_{\kmax}(\theta S)C_{\kmax+1}+
((1,\theta,\cdots,\theta^{\kmax})\otimes I_n)V_{\kmax+1},\label{eq:manip:Fkp1}
\end{equation}
and $F_{\kmax}$ is the first $\kmax$ columns of $F_{\kmax+1}$.

%

%

\begin{figure}[hb]
\begin{center}
\includegraphics{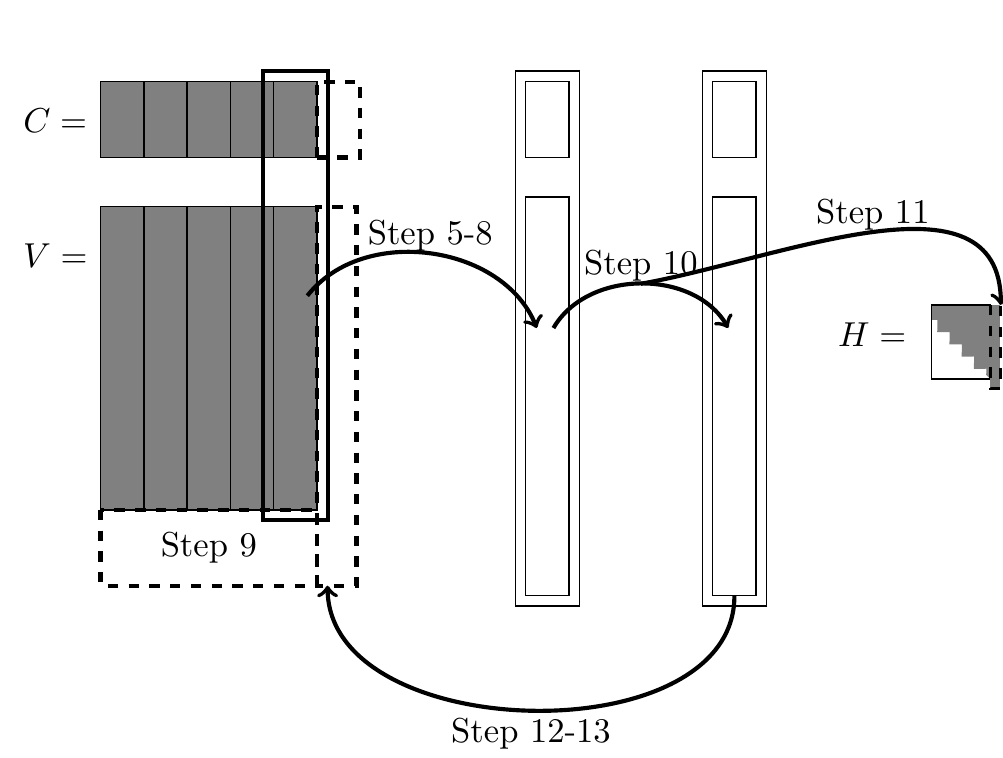}
\end{center}
\caption{Visualization of 
the infinite Arnoldi method with structured functions (Algorithm~\ref{alg:infarnoldi})} 
\label{fig:infarnoldi}
\end{figure}
\subsection{Gram-Schmidt orthogonalization}\label{sect:scalarprod}
%
%
\subsubsection{Computing the scalar product for structured functions}
%

For structured functions, i.e., functions of the form \eqref{eq:phistruct}, the consistent extension 
of the definition \eqref{eq:scalarprod_def0} is the following.
Let
\begin{equation}
  \varphi(\theta)= 
Y\exp_{N}(\theta S)c+
((1,\theta,\cdots,\theta^{N})\otimes I_n)x\label{eq:varphi_sp}
%
\end{equation}
and
\begin{equation}
  \psi(\theta)= 
Y\exp_{N}(\theta S)d+
((1,\theta,\cdots,\theta^{N})\otimes I_n)z.\label{eq:psi_sp}
\end{equation}
Then, 
\begin{equation}
 <\varphi,\psi>:=
\sum_{i=0}^Nz_i^Hx_i+\sum_{i=N+1}^{\infty}\frac{d^H(S^i)^HY^HY S^ic}{(i!)^2}
\label{eq:scalarprod_def}
\end{equation}
%
In practice, we can compute the scalar product by truncating the infinite sum 
and exploiting the structure of the sum.

\begin{lemma}[Computation of scalar product]
Suppose the two functions  
$\varphi:\CC\rightarrow\CC^n$ and $\psi:\CC\rightarrow\CC^n$ 
are given by
\eqref{eq:varphi_sp} and \eqref{eq:psi_sp}. Then
\begin{equation}
 <\varphi,\psi>=
\sum_{i=0}^Nz_i^Hx_i+
d^HW_{N+1,N_{\max}}c
+\veps_{N_{\max}},\label{eq:scalarprod_approx}
\end{equation}
with
\begin{equation}
W_{N,M}=\sum_{i=N}^{M}\frac{(S^i)^HY^HY S^i}{(i!)^2}
\label{eq:Wdef}
\end{equation}
provides an approximation to accuracy
\begin{equation}
  |\veps_{N_{\max}}|
\le
\|d\|_2\|Y^HY\|_2\|c\|_2
\frac{e^{2\|S\|_2}\|S\|_2^{2(N_{\max}+1)}}{((N_{\max}+1)!)^2}.
\label{eq:vepsapprox}
\end{equation}
\end{lemma}
\begin{proof}
By comparing the infinite sum \eqref{eq:scalarprod_def} 
with \eqref{eq:scalarprod_approx},
we can solve for $\veps_{N_{\max}}$ and bound the modulus,
\[
|\veps_{N_{\max}}|\le \|d\|_2\|Y^HY\|_2\|c\|_2 \sum_{i=N_{\max}+1}^{\infty}\frac{\|S\|_2^{2i}}{(i!)^2}
\le 
\|d\|_2\|Y^HY\|_2\|c\|_2
\left(\sum_{i=N_{\max}+1}^{\infty}\frac{\|S\|_2^{i}}{i!}\right)^2.
\]
The sum in the right-hand side can be interpreted as  the remainder term in the
Taylor approximation of $\exp(\|S\|_2)$.
The bound \eqref{eq:vepsapprox} follows by
applying Taylor's theorem.
\end{proof}

The lemma above has some properties important from a computational perspective:
\begin{itemize}
\item The sum in \eqref{eq:Wdef} involves only matrices of
size $p\times p$, i.e., it does not involve very large matrices, under the condition that  $Y^HY$ is precomputed.
\item The matrix $W_{N,M}$ defined by \eqref{eq:Wdef}
is constant if $S$ and $Y$ are constant. Hence, 
in combination with Algorithm~\ref{alg:infarnoldi} it only needs to be computed once in order to construct the Arnoldi factorization. 
\item An appropriate value of $N_{\max}$ such that 
$\varepsilon_{N_{\max}}$ is smaller than or comparable to
machine precision can be computed from $\|S\|$ by 
increasing $N_{\max}$ until the right-hand side of \eqref{eq:vepsapprox} 
is sufficiently small.
\end{itemize}

\subsubsection{Computing the Gram-Schmidt orthogonalization for structured functions}
%
%
One step of the Gram-Schmidt orthogonalization process
can be seen as a way of computing the \emph{orthogonal complement},
followed by normalizing the result.
When working with matrices, the process is compactly 
expressed as follows. 
Consider an orthogonal matrix $X\in\CC^{n\times k}$. The 
\emph{orthogonal complement} of a vector $u\in\CC$, with respect 
to the space spanned by the columns of $X$
and the Euclidean scalar product is given by,
\begin{equation}
  u_\orth=u-Vh.\label{eq:uorth}
\end{equation}
where 
\begin{equation}
  h=V^Hu.\label{eq:horth}
\end{equation}
In the setting of Arnoldi's method, the orthogonalization 
coefficients $h$ and the norm of the orthogonal complement
 $\beta$ needs to be returned to the Arnoldi algorithm.

\begin{algorithm}[h]
\caption[]{%
Gram-Schmidt orthogonalization for the scalar product \eqref{eq:scalarprod_def}\\
\phantom{\bf Algorithm 1:\;}$[c_\orth,x_\orth,h,\beta]=$\gramschmidt$(c,x,C,V)$
}\label{alg:taylorgm}
\begin{algorithmic}[1]
\INPUT  Vectors $c\in\CC^n$, $x\in\CC^{(N+1)n}$ representing the function
\[
\varphi(\theta):=Y\exp_{N}(\theta S)c+
((1,\theta,\cdots,\theta^{N})\otimes I_n)x
\] 
and $C\in\CC^{n\times k}$, $V\in\CC^{(N+1)n\times k}$, representing 
the block function, $F:\CC\rightarrow\CC^{n\times k}$, 
\[
 F(\theta)=
Y\exp_{N}(\theta S)C+
((1,\theta,\cdots,\theta^{N})\otimes I_n)V,
\] 
whose columns are orthogonal with respect to $<\cdot,\cdot>$ defined by \eqref{eq:scalarprod_def}.
\OUTPUT Orthogonalization coefficients $h\in\CC^{k}$,$\beta\in\CC$ and 
vectors $c_{\orth}\in\CC^{pn}$ and $x_{\orth}\in\CC^{(k+1)n}$ representing the normalized orthogonal complement of $\varphi$, 
\[
\varphi_\orth(\theta):=Y\exp_{N}(\theta S)c_\orth+
((1,\theta,\cdots,\theta^{N})\otimes I_n)x_\orth.
\]
\vspace{0.2cm}
\STATE
$h=V^Hx+C^H(W_{N+1,N_{\max}}c)$,
where $W_{N+1,N_{\max}}$ is given by \eqref{eq:Wdef}
\STATE
$c_{\orth}=c-Ch$
\STATE $x_{\orth}=x-Vh$
\STATE
$
g=V^Hx_\orth+C^H(W_{N+1,N_{\max}}c\orth)
$
\IF{ $\|g\|>$REORTH\_TOL}
\STATE
$c_{\orth}=c_\orth-Cg$
\STATE $x_{\orth}=x_\orth-Vg$
\STATE $h= h+g$
\ENDIF
\STATE
$\beta=x_\orth^Hx_\orth+c_\orth^H (W_{N+1,N_{\max}}c_\orth)$
\STATE $c_{\orth}=c_{\orth}/\beta$
\STATE $x_{\orth}=x_{\orth}/\beta$
\end{algorithmic}
\end{algorithm}

Due to the fact that the considered scalar product \eqref{eq:scalarprod_def} 
is the Euclidean scalar product on the Taylor coefficients, 
we can, similar to \eqref{eq:uorth} and \eqref{eq:horth}, compute
the orthogonal complement using matrices. 
The corresponding operations for our setting  are presented in the following theorem.
\begin{theorem}[Orthogonal complement]
Let 
$Y\in\CC^{n\times p}$,
$S\in\CC^{p\times p}$,
$C\in\CC^{p\times k}$,
$V\in\CC^{n(N+1)\times k}$ be the matrices representing the block function $F:\CC\rightarrow\CC^{n\times k}$, 
\[
  F(\theta)= Y\exp_{N}(\theta S)C+((1,\theta,\cdots,\theta^{N})\otimes I_n)V
\]
where the columns  are orthonormal with respect to $<\cdot,\cdot>$ defined by \eqref{eq:scalarprod_def}.  
Consider the function $\varphi$, represented by  $c_+\in\CC^p$ and $x_+\in\CC^{n(N+1)}$ and 
 defined by 
\[
\varphi(\theta)=Y\exp_{N}(\theta S)c_++
((1,\theta,\cdots,\theta^{N})\otimes I_n)x_+.
\]
and let $h\in\CC^k$, 
\[
  h:=V^Hx_++C^H\left(\sum_{i=N+1}^{\infty}\frac{(S^i)^HY^HYS}{(i!)^2}\right) c_+
\]
Then, the function $\varphi_\orth$, represented by the vectors
\[
c_\orth=c_+-Ch\in\CC^k,\;\; x_\orth=x_+-Vh\in\CC^{n(N+1)},
\]
and defined by 
\[
\varphi_\orth(\theta):=
Y\exp_{N}(\theta S)c_\orth+
((1,\theta,\cdots,\theta^{N})\otimes I_n)x_\orth
\]
is the orthogonal complement of $\varphi$ 
with respect to the space span by the the columns of $F$ and the scalar product $<\cdot,\cdot>$ defined by \eqref{eq:scalarprod_def}.
\end{theorem}
\begin{proof}
The construction is such that $\varphi_\orth$ is a linear combination of $\varphi$ and the columns of $F$ (due to linearity in coefficients $c$ and $x$). Remains to check that $\varphi_\orth$ is orthogonal to columns of $F$.
\end{proof}

The Gram-Schmidt process with reorthogonalization
can hence be efficiently implemented with operations on matrices and vectors. 
This is presented in 
Algorithm~\ref{alg:taylorgm}, where we used iterative
reorthogonalization  \cite{Bjorck:1994:GM} with (as usual) at most two steps.
In the numerical simulations we used REORTH\_TOL$=\sqrt{\varepsilon_{\rm mach}}$.

\section{Extracting and restarting}\label{sect:restart}
Recall the general outline described in Section~\ref{sect:linearprobs} 
and that 
we have now (in the Section~\ref{sect:lockedinfarn}) described the first  step
in detail.
In what follows we discuss the second step. 
We propose a procedure to carry out some operations of the result
of the first step, i.e., Algorithm~\ref{alg:infarnoldi}, and restart
it such that we expect that the outer iteration eventually
converges to a partial Schur factorization.

%
%
%

%
%
\subsection{Manipulations of the Arnoldi factorization}\label{sect:manips}
First recall that Algorithm~\ref{alg:infarnoldi} 
is an Arnoldi method in a function setting and the
output corresponds to an Arnoldi factorization, 
\begin{equation}
     (\BBB F_k )(\theta)=F_{k+1}(\theta) 
\underline{H}_k,\label{eq:arnfact2}
\end{equation}
where, the block function $F_{k+1}$ is given 
by the output of Algorithm~\ref{alg:infarnoldi} with the
defintion
\begin{equation}
 F_{k+1}(\theta)=
Y\exp_{k}(\theta S)C_{k+1}+
((1,\theta,\cdots,\theta^{k})\otimes I_n)V_{k+1},\label{eq:manip:Fkp1}
\end{equation}
and $F_{k}$ is the first $k$ columns of
$F_{k+1}$. To ease the notation, we have denoted $k=\kmax$. 

Although the Arnoldi factorization \eqref{eq:arnfact2} 
is a function relation, we will now see that
several parts of the steps
for implicit restarting (cf. 
\cite{Morgan:1996:RESTARTING,Lehoucq:2001:IRAM,Lehoucq:1996:DEFLATION,Stewart:2001:KRYLOVSCHUR}) for Arnoldi's method (for linear matrix 
eigenvalue problems)
can be carried out in a similar way by  working with functions.

We will start by computing an ordered Schur factorization of ${H}_k$
\begin{equation}
    Q^*H_kQ=
    (Q_1,Q_2,Q_3)^*{H}_k(Q_1,Q_2,Q_3)=
\begin{pmatrix}
R_{11} & R_{12} & R_{13}\\
       & R_{22} & R_{23}\\
       &        & R_{33}\\
\end{pmatrix}\label{eq:orderedschur}
\end{equation}
where $R_{11}\in\CC^{p_l\times p_l}$,
$R_{22}\in\CC^{(p-p_l)\times(p-p_l)}$ and
$R_{33}\in\CC^{(k-p)\times(k-p)}$ are  upper triangular
matrices. 
The ordering is such that 
the eigenvalues of $R_{11}$ are 
very accurate (and from now on called the locked Ritz values), 
the eigenvalues of $R_{22}$ are
wanted eigenvalues (selected according to some criteria)
 which have not converged,
and the eigenvalues of $R_{33}$ are unwanted.

Hence,
\begin{equation}
\begin{pmatrix}
Q^* & \\
   & 1
\end{pmatrix}
\underline{H}_k(Q_1,Q_2,Q_3)=
\begin{pmatrix}
R_{11} & R_{12} & R_{13}\\
       & R_{22} & R_{23}\\
       &        & R_{33}\\
a_1^T  & a_2^T  & a_3^T
\end{pmatrix}.\label{eq:orderedschur2}
\end{equation}
Note that $a_1$ is a measure of the (unstructured) backward error
of the corresponding eigenvalues of $R_{11}$ and $\|a_1\|$ is often
used as stopping criteria. Hence,
$\|a_1\|$ will be zero if the eigenvalues of $R_{11}$ are exact and 
will in general be small (or very small) relative to 
$\underline{H}_k$  
since the eigenvalues of $R_{11}$ are very accurate solutions.
By successive application of Householder reflections (see e.g. \cite{Meerbergen:2008:QUADARNOLDI}) 
we can now construct an 
 orthogonal matrix $P_2$ such that 
\begin{equation}
\begin{pmatrix}
I_{p_l} && \\ & P_2^* & \\ & &1
\end{pmatrix}
\begin{pmatrix}
R_{11} & R_{12} & \\
       & R_{22} & \\
a_1^T  & a_2^T  &
\end{pmatrix}
\begin{pmatrix}
I_{p_l} & \\ & P_2  
\end{pmatrix}=
\begin{pmatrix}
R_{11} & M \\
0      & \hat{H} \\
a_1^T    & e_{p-p_l}^T \beta
\end{pmatrix},\label{eq:backhessenberg}
\end{equation}
where 
\[
  \underline{\hat{H}}:=
\begin{pmatrix}
\hat{H}\\
e_{p-p_l}^T\beta
\end{pmatrix}
\]
is a Hessenberg matrix.

By considering the leading two blocks and columns of 
\eqref{eq:orderedschur2} and the result of the Householder reflection
transformation \eqref{eq:backhessenberg} we find that 
\begin{equation}
\begin{pmatrix}
(Q_1,Q_2P_2)^* & \\
& 1
\end{pmatrix}
\underline{H}_k(Q_1,Q_2P_2)=
\begin{pmatrix}
R_{11} & Z \\
0      & \hat{H} \\
a_1^T    & e_{p-p_l}^T \beta
\end{pmatrix}=
\begin{pmatrix}
R_{11} & Z \\
0      & \underline{\hat{H}} 
\end{pmatrix}+O(\|a_1\|).\label{eq:backhessenberg2}
\end{equation}
These operations yield a transformation of the Arnoldi factorization
where the first block is triangular (to order $O(\|a_1\|)$). 
%
We reach the following result, which is an
Arnoldi factorization similar to \eqref{eq:arnfact2}
but only of length $p$. Moreover, the Hessenberg matrix  
does not contain the unwanted eigenvalues and has
a leading block which is almost triangular. 
\begin{theorem}\label{thm:arnfactp}
Consider an Arnoldi factorization given by \eqref{eq:arnfact2}
and let $F_{k+1}(\theta)=(F_k(\theta),f(\theta))$. 
Let $Q_1$ and $Q_2$ represent the leading blocks in the ordered Schur decomposition 
\eqref{eq:orderedschur} and let $P_2$, $R_{11}$ and $\underline{\hat{H}}$
be the result of the Householder reflections in \eqref{eq:backhessenberg2}. Moreover, 
let
\begin{equation}
   G_{p}(\theta):=F_k(\theta)(Q_1,Q_2P_2),\;\; G_{p+1}(\theta):=(G_p(\theta),f(\theta)).\label{eq:Gpdef}
\end{equation} 
Then, $G_{p+1}$ approximately satisfies the length $p<k$  Arnoldi factorization 
\begin{equation}
  (\BBB G_p)(\theta)=G_{p+1}(\theta)
\begin{pmatrix}
R_{11} & Z \\
0      & \underline{\hat{H}} 
\end{pmatrix}+O(\|a_1\|).\label{eq:arnfactp}
\end{equation}
\end{theorem}
\subsection{Extraction and imposing structure}
%
%
Restarting in standard IRAM for matrices essentially
consists of 
assigning the algorithmic state of the Arnoldi method to that 
corresponding to the factorization in Theorem~\ref{thm:arnfactp}. 
The direct adaption of this procedure is not suitable  in our
setting 
due to a growth of the polynomial part of the structured functions. This can be seen as follows.
Suppose we start Algorithm~\ref{alg:arnoldiBBB}  with
a constant function (as done in \cite{Jarlebring:2010:TRINFARNOLDI}) 
and carry out the construction of $G_p$ as in Theorem~\ref{thm:arnfactp}. 
Then, $G_{p+1}$ will be a matrix with polynomials of degree $k$. 
We hence need to start with a state consisting of polynomials of degree $k$. 
The degree of the polynomial will grow with each restart
and after $M$ restarts, 
the polynomials will be  of degree $Mk$. The representation 
of this polynomial 
will hence quickly limit the efficiency of the restarting scheme.

Instead of restarting with polynomials we will
perform an explicit restart using Algorithm~\ref{alg:infarnoldi}
with a particular choice of the input which we here denote
$\hat{Y}$, $\hat{S}$, $\hat{c}$.
This choice is 
\emph{inspired} by the factorization in Theorem~\ref{thm:arnfactp}.

We will first impose exponential structure on $G_p$ 
in the sense that we consider a function $\hat{G}_p$, with the property
\[
  G_p(0)=\hat{G}_p(0)
\]
and defined by
\begin{equation}
\hat{G}_p(\theta):=G_p(0)\exp(\hat{S}\theta) \label{eq:impstruct}
\end{equation}
where
\begin{equation}
\hat{S}=
\begin{pmatrix}
R_{11} & Z \\
0 & \hat{H}
\end{pmatrix}^{-1}.\label{eq:setS}
\end{equation}
Note that we can express $G_p(0)$ explicitly from \eqref{eq:Gpdef} as
\begin{equation}
   G_p(0)=(V_{k+1,1}Q_1,\; V_{k+1,1}Q_2P_2)=:\hat{Y}.\label{eq:setY}
\end{equation}
where $V_{k+1,1}$ is the upper $n\times (k+1)$-block of $V_k$.

Assume for the moment that $\|a_1\|=0$. Then,
the first $p_l$ columns of \eqref{eq:arnfactp} 
correspond to the definition of an invariant pair $(\Psi,R)$, where 
$\Psi(\theta)=G_{p_l}(\theta)$ and $R=R_{11}$. From
Theorem~\ref{thm:invpairs} we know that $\Psi$ is of exponential structure, 
and imposing the structure as in \eqref{eq:impstruct} does
not modify the function, i.e., if $\|a_1\|=0$, then $\hat{G}_{p_l}(\theta)=G_{p_l}(\theta)$.
 Hence, 
the first $p_l$ columns of the equation \eqref{eq:arnfactp} are 
preserved also if we replace $G_p(\theta)$ with 
$\hat{G}_p(\theta)$. Due to the fact that $\|a_1\|$ is 
small (or very small) we expect that imposing the structure as 
in \eqref{eq:impstruct} gives an approximation of the $p_l$ columns
of \eqref{eq:arnfactp}, i.e., 
\begin{equation}
(\BBB \hat{G}_{p_l})(\theta)\approx
\hat{G}_{p_l+1}(\theta)
\begin{pmatrix}
R_{11}\\
0
\end{pmatrix},\label{eq:lockedfact}
\end{equation}
if $\|a_1\|$ is small and equality is achieved if $\|a_1\|=0$.

%

\begin{algorithm}[h]
\caption[]{Structured explicit restarting with locking\\
\phantom{\bf Algorithm 1:}$[S,Y]=$\tt{infarn\_restart}$(x_0,\lambda_0,\kmax,p)$
}\label{alg:exprestart}
\begin{algorithmic}[1]
\INPUT $x_0\in\CC^n$, $\lambda_0$ representing the function
\[
  f(\theta)=\exp(\lambda_0\theta)x_0,
\]
  maximum size of subspace $\kmax$, number of wanted eigenvalues $p$ 
\OUTPUT $S,Y$ such that $(Y,S)$ represents  an invariant pair
\vspace{0.2cm} 
\STATE Normalize $f$ by setting $x_0=\frac{1}{\|x_0\|\sqrt{W_{0,N_{\max}}}}x_0$, with 
$W_{0,N_{\max}}$ is given by \eqref{eq:Wdef} with $S=\lambda_0$
\STATE Set $Y_0=(x_0,0,\ldots,0)\in\CC^{n\times p}$.
\STATE Set $S=\diag(\lambda_0,1,\ldots,1)\in\CC^{p\times p}$
\STATE Set $c=e_1\in\CC^{p}$, $p_l=0$
\WHILE{$p_l<p$}
\STATE $[V,C,\underline{H}_{\kmax}]=$\infarnoldiexp$(c,S_j,Y_j,p_l,\kmax)$
\STATE For  every eigenvalue of $H_{\kmax}$ classify it as, lock, wanted or unwanted, and let $p_l$ denote the number of locked eigenvalues.
\STATE Compute ordered Schur factorization of $H_{\kmax}$ partitioned according to \eqref{eq:orderedschur}
\STATE Compute the $a_2$ vector in \eqref{eq:orderedschur2}
\STATE Compute the orthogonal matrix $P_2$ according to \eqref{eq:backhessenberg}
\STATE Compute $Z$ and $\hat{H}$ from \eqref{eq:backhessenberg2}
\STATE Set $Y_{j+1}=\hat{Y}$ and $S_{j+1}=\hat{S}$ according to \eqref{eq:setY} and \eqref{eq:setS} 
\STATE Reorthogonalize the function $F(\theta)=Y_{j+1}\exp(\theta S_{j+1})(e_1,\ldots,e_{p_l})$
\STATE $[c,\cdot,\cdot,\cdot]=$\gramschmidt$(e_{p_l+1},\cdot,C_{j+1},\cdot)$
\STATE Set $j=j+1$
\ENDWHILE
\end{algorithmic}
\end{algorithm}

With the above reasoning we have a justification to use the first $p_l$ 
columns of \eqref{eq:impstruct}, i.e., 
\[
  \hat{Y}\exp(\hat{S}\theta)\begin{pmatrix}I_{p_l}\\0\end{pmatrix}
\]
in the initial state for the restart. In the approximation
of the $p-p_l$ last columns of $G_p$ by the $p-p_l$ last columns 
of \eqref{eq:impstruct}, the Arnoldi relation in the function setting is
in general lost, because Ritz functions only have
exponential structure upon convergence. Therefore, we will only
use the 
$(p_l+1)$st column in the restart, from
which the Krylov  space will be extended again in the next inner iteration
This leads us to a restart with the function
\[
   \hat{Y}\exp(\hat{S}\theta)\begin{pmatrix}
I_{p_l+1}\\ 0
\end{pmatrix},
\]
which corresponds to setting $\hat{c}=e_{p_l+1}$ and initial function
\[
  f(\theta)=\hat{Y}\exp(\hat{S}\theta)e_{p_l+1}.
\]

By these modifications of the factorization \eqref{eq:arnfactp} we 
have now reached a choice of $\hat{Y}$ given by 
\eqref{eq:setY}, $\hat{S}$ given by \eqref{eq:setS} 
and $\hat{c}=e_{p_l+1}$. This choice of variables
satisfy all the properties necessary for the 
input of Algorithm~\ref{alg:infarnoldi}, except the orthogonality condition. 
The first columns of $\hat{G}_{p_l}$ are automatically 
orthogonal (at least if $\|a_1\|=0$). 
The $(p_l+1)$st column will however in general
 not be orthogonal to $\hat{G}_{p_l}$, which is an assumption needed 
for Algorithm~\ref{alg:infarnoldi}.
It is fortunately here
easily 
remedied by orthogonalizing the function corresponding to $\hat{c}=e_{p_l+1}$
using the function \gramschmidt, i.e., Algorithm~\ref{alg:taylorgm}.

The details of this selection as well as the manipulations in 
Section~\ref{sect:manips} are summarized in the outer iteration 
Algorithm~\ref{alg:exprestart}.

\begin{remark}[Explicit restart without locking]
Note that a restart which is theoretically
very similar to what we have here proposed, can be achieved by
starting the infinite Arnoldi method with the function of
the first column of \eqref{eq:impstruct}, without taking the ``locked part'' of
the factorization directly into account. 
Such an explicit restarting technique (without locking) 
does unfortunately have unfavorable numerical properties and will not 
be persued here.
%
From reasoning similar to \cite{Morgan:1996:RESTARTING} we know
that the first column of \eqref{eq:impstruct} is an approximation
of an element of an invariant subspace and the first $p_l$ 
steps of Arnoldi's method started with this vector is expected to recompute 
the $p_l$ converged Ritz vectors after $p_l$ iterations.
In the $(p_l+1)$st iteration, the 
Arnoldi vector is corrupted due to cancellation.
\end{remark}

\section{Examples}\label{sect:examples}

\subsection{A small example of Hadeler}\label{sect:hadeler}
The nonlinear eigenvalue problem presented in \cite{Hadeler:1967:MEHRPARAMETERIGE}, 
which is available with the name {\tt hadeler} in
the problem collection \cite{Betcke:2010:NLEVPCOLL}, is given by
\[
   M(\lambda)=-A_0+(\lambda+\mu)^2A_1+(e^{\lambda+\mu}-1)A_2,
\]
where $A_i\in\RR^{n\times n}$, $i=0,\ldots,2$ with $n=8$ and $\mu$ is
a shift which we will use to select a point close to which we will
find the eigenvalues.

In order to apply Algorithm~\ref{alg:infarnoldi} we need to derive 
a formula for $x_{+,0}$ in \eqref{eq:y0}. 
The derivatives for $M$ are straightforward to compute and we compute  
$\MM_N$, using
\eqref{eq:MMNdef} and
\eqref{eq:MMNdef2}. More precisely, we use the following computational expressions, 
\begin{eqnarray*}
\MM_{-1}(Y,S)c_+&=&
-A_0Yc_++A_1Y(S+\mu I)^2c_++A_2Y(\exp(S+\mu I)-I)c_+,\\
\MM_{0}(Y,S)c_+&=&
A_1Y(S^2+2\mu S)c_++e^{\mu}A_2Y(\exp(S)-I)c_+\\
\MM_{1}(Y,S)c_+&=&
A_1Y(S^2c_+)+e^\mu A_2Y(\exp(S)-I-S)c_+\\
\MM_{N}(Y,S)c_+&\approx&
e^\mu A_2Y
\left(
\sum_{i=N+1}^{i_{\max}}
\frac{S^ic_+}{i!}\right), \;\;N>1
\end{eqnarray*}
In the last formula, $i_{\max}$ is chosen such that the expression has
converged to machine precision. Since this is not computationally
expensive, we can roughly overestimate $i_{\max}$. In this  example 
it was sufficient to take  $i_{\max}=40$.

In the outer algorithm (Algorithm~\ref{alg:exprestart}) 
we classified a Ritz value as converged (locked) 
when the absolute residual was smaller than $1000\times \epsm$. 
We selected the largest eigenvalues of $H_k$ as the wanted eigenvalues.

The convergence is illustrated for two runs 
in Figure~\ref{fig:combined_hadeler2}   and Figure~\ref{fig:combined_hadeler3}. In order to illustrate the similarity with  implicit restarting
in \cite{Sorensen:1992:IMPLICIT}, 
we also carried out the infinite Arnoldi 
method
with true implicit restarting by restarting only with polynomials, instead
of using Algorithm~\ref{alg:exprestart}. We clearly see that at least in 
the beginning of the iteration, the convergence of 
Algorithm~\ref{alg:exprestart} is similar to the convergence of 
IRAM. Note that
IRAM in this setting exhibits a growth of the basis matrix
 and it is hence 
considerably slower. We show the number of locked Ritz-values Table~\ref{tbl:indicators}. 
Moreover, we quantify the impact of the procedure to impose the structure in the restart by inspecting the approximation in \eqref{eq:lockedfact}.
We define $\gamma$  as the norm of the difference of the left and right-hand side of \eqref{eq:lockedfact}. Lemma~\ref{thm:structinvres} 
shows that this difference
is independent of $\theta$ and provides a computable expression. More precisely, 
\begin{multline}
\gamma:=
\left\|(\BBB \hat{G}_{p_l})(\theta)-\hat{G}_{p_l+1}(\theta)
\begin{pmatrix}R_{11}\\0 \end{pmatrix}\right\|_2=%
\\
\left\|(\BBB \hat{G}_{p_l})(\theta)-\hat{G}_{p_l}(\theta)
R_{11}\right\|_2
=\| M(0)^{-1}\MM(Y,S)S^{-1}\|_2,
\end{multline}
where we used that $\hat{G}_{p_l}$ has the structure $\hat{G}_{p_l}(\theta)=Y\exp(\theta R_{11}^{-1})$.
The values of $\gamma$ are also given in Table~\ref{tbl:indicators}. 
They are, as expected, of the same order of magnitude as the locking tolerance.

\begin{table}[t]
\begin{center}
\begin{tabular}{c|c|c||c|c} 
&%
\multicolumn{2}{c||}{Run 1} &
\multicolumn{2}{c}{Run 2} \\\hline
Outer iteration &$p_l$ & $\gamma$ &  $p_l$&$\gamma$\\
\hline
1 &0 &   &  0 & 0   \\
2 &1 &$1.0\times 10^{-15}$   &  0 & 0\\ 
3 &2 &$5.7\times 10^{-14}$   &  3 & $6.4\times 10^{-14}$ \\
4 &2 &$5.7\times 10^{-14}$   &  3 & $6.4\times 10^{-14}$ \\
5 &3 &$5.8\times 10^{-14}$   &  3  &$1.4\times 10^{-14}$ \\
6 &3 &$5.8\times 10^{-14}$   &  4  &$1.4\times 10^{-14}$ \\
7 &4 &$7.3\times 10^{-13}$  &  5  &$1.4\times 10^{-14}$ \\
8 &10&$2.3\times 10^{-13}$  & &\\ 
\end{tabular}
\end{center}
\caption{The indicator value and number of locked Ritz values for the two runs 
of the example of Hadeler in Section~\ref{sect:hadeler}. 
Run 1 corresponds to Figure~\ref{fig:combined_hadeler2} and Run 2 corresponds to 
Figure~\ref{fig:combined_hadeler3}. The outer iteration count represents the number of loops carried out in Algorithm~\ref{alg:exprestart}.
}
\label{tbl:indicators}
\end{table}

\begin{figure}[ht]
\begin{center}
\scalebox{0.9}{\includegraphics{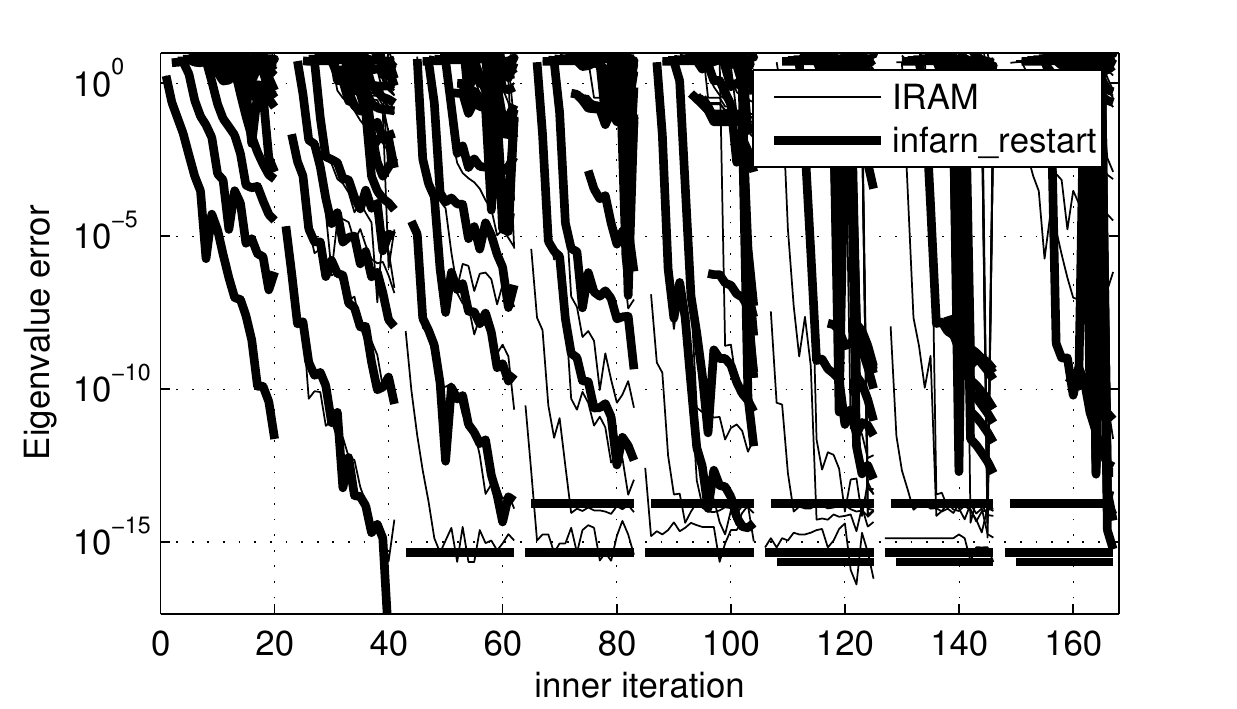}}
\end{center}
\caption{Convergence of Algorithm~\ref{alg:exprestart} (thick) and 
implicitly restarted Arnoldi \cite{Sorensen:1992:IMPLICIT} (thin)
($\kmax=20$, $p=10$, $\mu=-1$) for the Hadeler example in Section~\ref{sect:hadeler}}
\label{fig:combined_hadeler2}
\end{figure}

\begin{figure}[ht]
\begin{center}
\scalebox{0.9}{\includegraphics{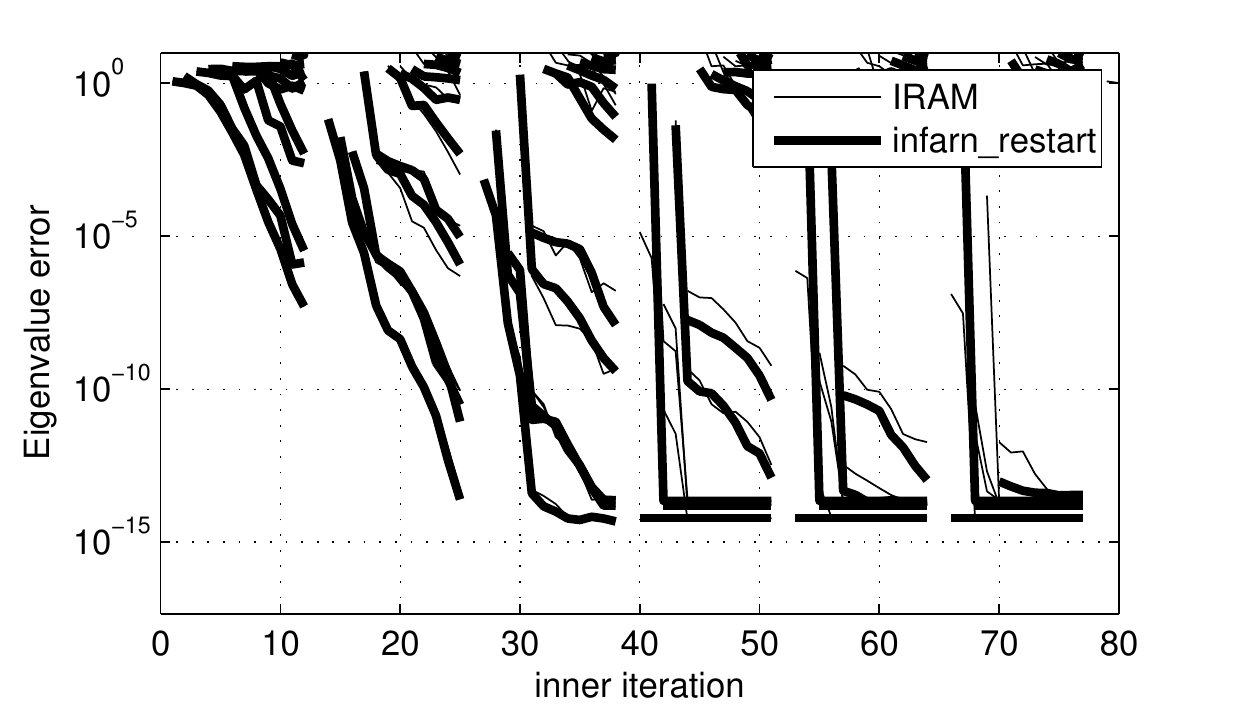}}
\end{center}
\caption{Convergence of Algorithm~\ref{alg:exprestart} (thick) and 
implicitly restarted Arnoldi \cite{Sorensen:1992:IMPLICIT} (thin)
($\kmax=12$, $p=5$, $\mu=3+5\iota$) for the Hadeler example in Section~\ref{sect:hadeler}}
\label{fig:combined_hadeler3}
\end{figure}

\begin{figure}[ht]
\begin{center}
\scalebox{0.9}{\includegraphics{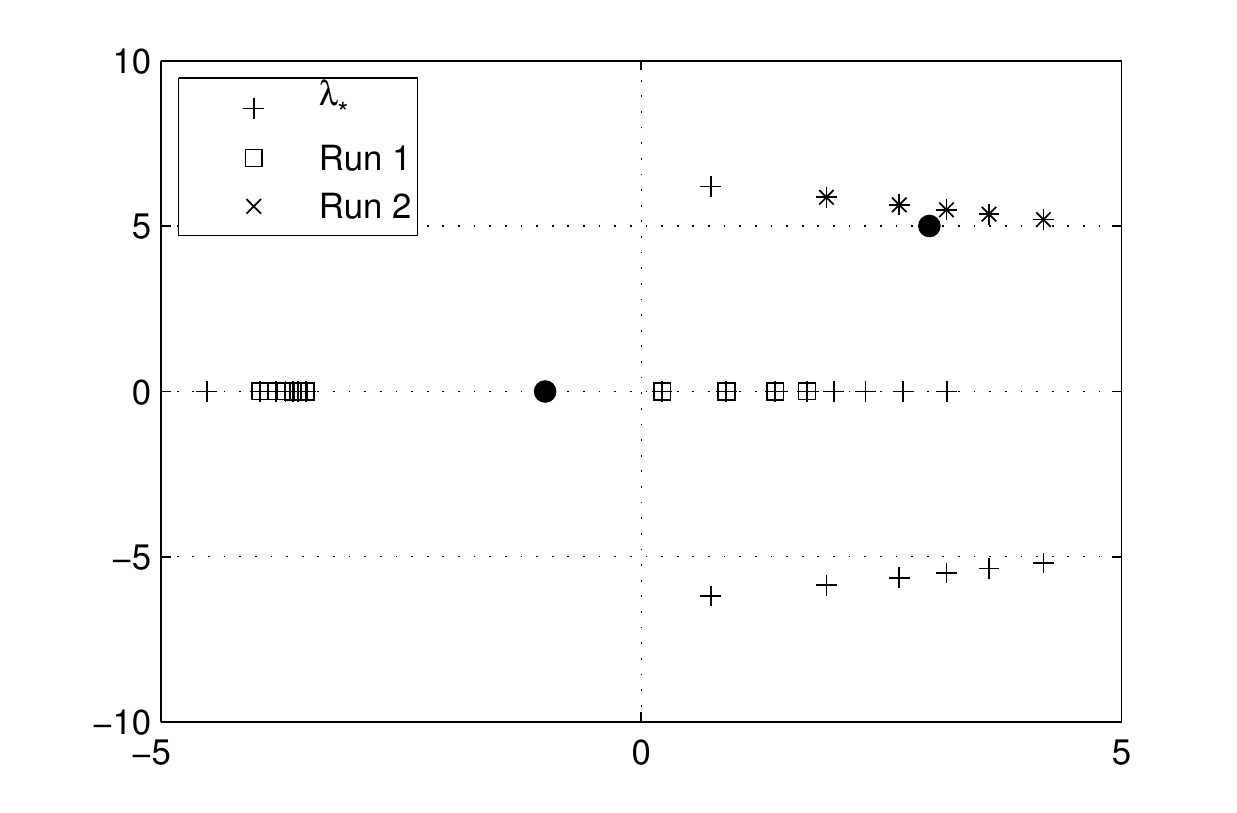}}
\end{center}
\caption{Computed eigenvalues and shifts for the Hadeler example}
\label{fig:combined_hadeler_eigs}
\end{figure}

\subsection{A large-scale square-root example}\label{sect:sqrt}
We considered the same example as in  \cite[Section~7.2]{Jarlebring:2010:TRINFARNOLDI},
which is the problem called {\tt gun} in the problem collection \cite{Betcke:2010:NLEVPCOLL}
and stems
 from \cite{Liao:2006:SOLVING}. 
It is currently the largest example, among those examples 
in  the collection \cite{Betcke:2010:NLEVPCOLL}
which are neither 
polynomial eigenvalue problems  nor rational eigenvalue problems.

In order to focus on a particular region in the complex plane we
introduce (as in \cite{Jarlebring:2010:TRINFARNOLDI}) a shift $\mu$
and a scaling $\gamma$, for which the nonlinear eigenvalue problem is 
\[
  M(\lambda)=
A_0
-(\gamma\lambda+\mu)A_1+
\iota\sqrt{\gamma\lambda+\mu-\sigma_1^2}A_2+ 
\iota\sqrt{\gamma\lambda+\mu-\sigma_2^2}A_3
\]
where $\sigma_1=0$ and $\sigma_2=108.8774$ and  $\iota^2=-1$. We selected
$\gamma=300^2-200^2$ and $\mu=250^2$ since this transforms the region
of interest to be essentially within the unit circle.

%
%
In order to compute a formula for $x_{+,0}$ in \eqref{eq:y0}, we 
need in particular
\begin{multline}
  \MM(Y,S)=  
A_0Y
-A_1Y(\gamma S+\mu I_p)+\notag\\
\iota A_2Y\sqrt{\gamma S+(\mu-\sigma_1^2)I_p} +
\iota A_3Y\sqrt{\gamma S+(\mu-\sigma_2^2)I_p} \label{eq:sqrtMMN}
\end{multline}
where $\sqrt{Z}$ denotes the matrix square root (principal branch).

%

We will partially base  the formulas on  
the Taylor coefficients of the square root in order to compute
$\MM_N$ (needed in the
computation of $x_{0,+}$ in \eqref{eq:y0}). We will use
\[
  \sqrt{\gamma\lambda+\mu-\sigma_j^2}=
\alpha_{0,j}+
\alpha_{1,j}\lambda+
\alpha_{2,j}\lambda^2+\cdots
\]
where
\begin{subequations}\label{eq:sqrtcoeffs}
\begin{eqnarray}
  \alpha_{0,j}&=&\sqrt{\mu-\sigma_j^2}\\
  \alpha_{k,j}&=&
\left(\frac{\gamma}{2}\right)
\left(-\frac{\gamma}{2}\right)
\left(-\frac{3\gamma}{2}\right)
\cdots
\left(-\frac{(2k-3)\gamma}{2}\right)(\mu-\sigma_j^2)^{1/2-k},\;\; k>0.
\end{eqnarray}
\end{subequations}
%
\begin{table}[t]
\begin{center}
\begin{tabular}{c|c|c|c} 
 & $\kmax=50$      & $\kmax=30$ & $\kmax=25$\\  \hline\hline
nof. restarts    & 0       & 1 & 3  \\
\hline
total CPU    & 35.7s       & 21.0s & 23.7s  \\
LU decomp.   & 2.1s        & 2.1s  &  2.1s \\
\gramschmidt & 23.7s       & 12.9s &  15.1s  \\
computing $x_{+}$  &  6.8s& 6.9s &  3.4s \\
\hline
Memory  usage & $\sim 200~\mathrm{MB}$ &  $\sim 78~\mathrm{MB}$ & $\sim 58~\mathrm{MB}$   \\
\end{tabular}
\end{center}
\caption{Consumption of memory resources and profiling times, for some
choices of the restart parameter $\kmax$ and $p=10$. Memory in megabytes (MB) and CPU time in seconds. 
}
\label{tbl:sqrt_resources}
\end{table}

This can be used to compute of $\MM_N$, as follows,
\begin{subequations}\label{eq:sqrtMM}
\begin{eqnarray}
  \MM_{0}(Y,S)c_+  &=&  \MM(Y,S)c_+- M(0)Yc_+ \label{eq:sqrtMMN0}\\
  \MM_{1}(Y,S)c_+  &=&  \MM_{0}(Y,S)c_+-M'(0)Y(Sc_+) \label{eq:sqrtMMN1}\\
  \MM_{N}(Y,S)c_+  &=&  
\sum_{i=N+1}^{\infty}\frac{1}{i!}M^{(i)}(0)YS^ic_+  \notag \\
&\approx& \iota A_2\left(Y \sum_{i=N+1}^{i_{\max}}\alpha_{i,1}S^ic_+\right) 
+ \iota A_3\left(Y\sum_{i=N+1}^{i_{\max}}\alpha_{i,2}S^ic_+\right)%
,\;\;
N>1.\;\;\;\;\;\;\;\;\;\;\;\;\label{eq:sqrtMMNp}
\end{eqnarray}
\end{subequations}
%
Note that the sums in \eqref{eq:sqrtMMNp} are operations with vectors
of relatively small dimension  and can be computed efficiently. We selected the
number of terms  $i_{\max}$ adaptively such that $\|S^{i_{\max}}c_+\||\alpha_{i_{\max},k}|\ll\varepsilon_{\rm mach}$.

This results in the following formulas which we used for the computation of $x_{+,0}$
\[
x_{+,0}
=-M(0)^{-1}(\MM_{0}(Y,S)c_+), \textrm{ for }N=0
\]
\[
x_{+,0}
=-M(0)^{-1}(\MM_1(Y,S)c_++M'(0)x_{+,1}), \textrm{ for }N=1,
\]
and for $N>1$, 
\[
x_{+,0}
=-M(0)^{-1}\left(\MM_N(Y,S)c_++
\iota A_2\sum_{j=1}^Nx_{+,j}(\alpha_{j,1} (j!)) + 
\iota A_3\sum_{j=1}^Nx_{+,j}(\alpha_{j,2}(j!))\right).
\]
The matrix $M(0)$ was factorized (with an LU-factorization) before starting the iteration, such that
$M(0)^{-1}b$  could be computed efficiently.

We first wish to illustrate that the restarting and structure exploitation 
can considerably reduce both memory and CPU usage.
In Table~\ref{tbl:sqrt_resources} we compare runs for 
the standard version 
of the infinite Arnoldi method \cite{Jarlebring:2010:TRINFARNOLDI} 
(first column) with the restarting algorithm
(Algorithm~\ref{alg:exprestart}) for two choices of the parameter
$\kmax$. The iteration was terminated when $p=10$ eigenvalues were
found. We clearly see that for the choices of $\kmax$ there is a
considerable reduction in memory and some reduction in computation
time. 
%

In Figure~\ref{fig:sqrt_p9} and Figure~\ref{fig:sqrt_p14} we 
illustrate that the algorithm scales reasonably well with $p$, 
i.e., the number of wanted eigenvalues. When we increase $p$, we need
more outer iterations, but eventually the algorithm usually converges
for reasonably large $p$.

\begin{figure}[ht]
\begin{center}
\scalebox{0.9}{\includegraphics{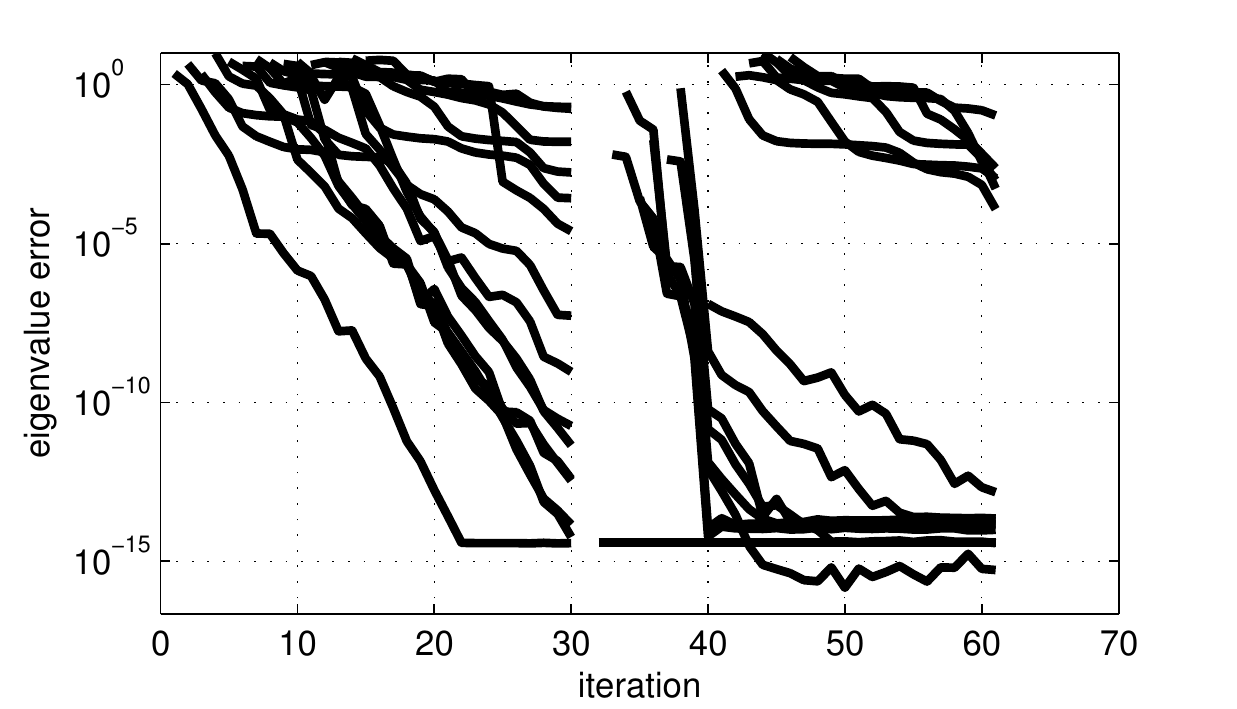}}
\end{center}
\caption{Convergence history for Algorithm~\ref{alg:exprestart} with
  the  example involving a square root in Section~\ref{sect:sqrt} ($p=9$)}
\label{fig:sqrt_p9}
\end{figure}
\begin{figure}[ht]
\begin{center}
\scalebox{0.9}{\includegraphics{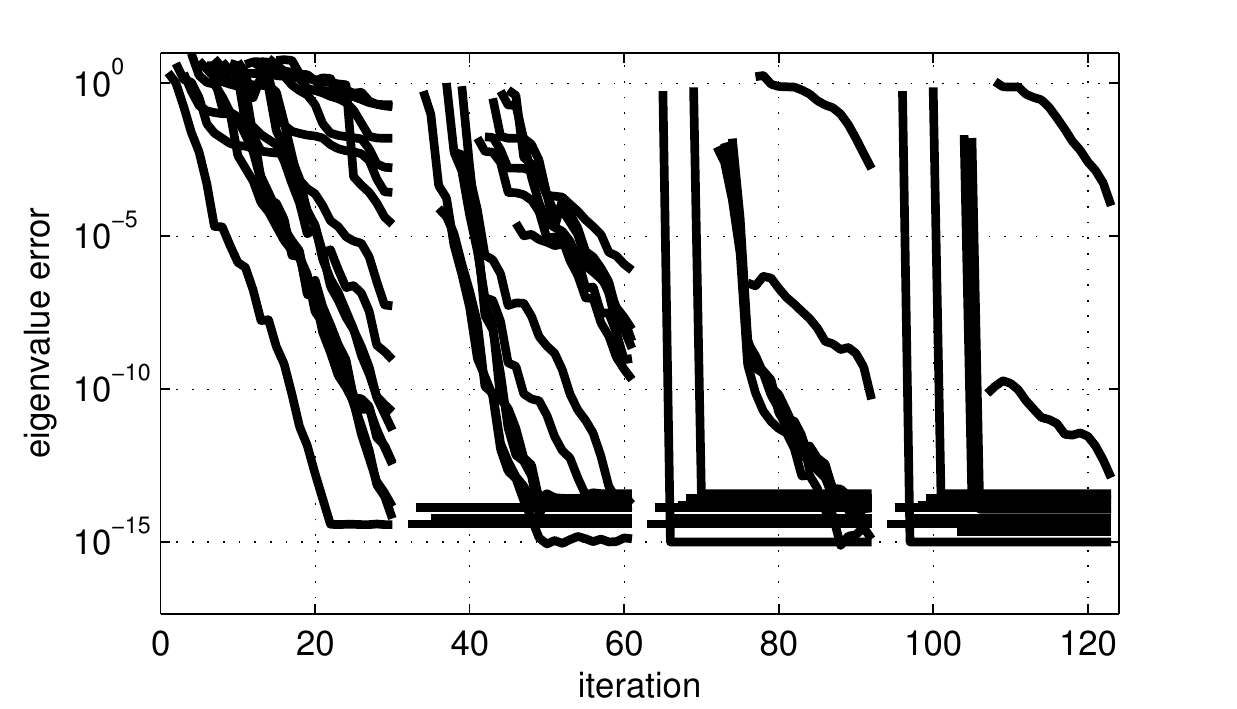}}
\end{center}
\caption{Convergence history for Algorithm~\ref{alg:exprestart} with
  the  example involving a square root in Section~\ref{sect:sqrt} ($p=14$)}
\label{fig:sqrt_p14}
\end{figure}

%
%
%
\section{Concluding remarks}
We have in this work shown how the partial Schur factorization
of an operator $\BBB$ can be computed using a variation
of the procedures to compute partial Schur factorization
for matrices. 
Several variations of the results for matrices
appear to be possible to adapt. Concepts like thick restarting,
purging and other selection  
strategies, appear to carry over but deserve further attention.
We also wish to point that many of the
results allow to be adapted or used in other ways. 
In this paper we also presented a
Taylor-like scalar product, which could, essentially be replaced 
by any suitable scalar product.

%

\bibliographystyle{plain}
\bibliography{fullbib}

\appendix
\section{A technical lemma}

\begin{lemma}\label{thm:structinvres}
Consider  $Y\in\CC^{n\times p}$  and $S\in\CC^{p\times p}$, where $S$ is invertible. 
Let $F(\theta):=Y\exp(\theta S)$. Then, 
\begin{equation}
(\BBB F)(\theta)-F(\theta)S^{-1}=
-M(0)^{-1}\MM(Y,S)S^{-1}.\label{eq:structinvres}
\end{equation}
\end{lemma}
\begin{proof}
We prove the theorem by showing that the derivative of the function
relation \eqref{eq:structinvres} holds for any $\theta$ and that the
relation holds in one point $\theta=0$. Note that the right-hand side
of \eqref{eq:structinvres} is  constant (with respect to  $\theta$) 
and the derivative of the left-hand side reduces to 
\[
   F(\theta)-F'(\theta)S^{-1}=F(\theta)-Y\exp(\theta S)SS^{-1}=0,
\]
by definition of $\BBB$ and differentiation of $\exp(\theta S)$.

From the definition of $\BBB$ and evaluation of the left-hand side of \eqref{eq:structinvres} 
 at $\theta=0$ we have, 
\begin{equation}
   (\BBB F)(0)-F(0)S^{-1}=\left(B(\frac{d}{d\theta})Y\exp(\theta S)\right)(0)-YS^{-1}.\label{eq:BBBF0}
\end{equation}
Note that for an analytic scalar  function $b:\CC\rightarrow\CC$, 
\[
  \left(b(\frac{d}{d\theta})\exp(\theta S)\right)(0)=b(S).
\]
Hence,
\begin{multline*}
\left(B(\frac{d}{d\theta})Y\exp(\theta S)\right)(0)=
B_1Y(b_1(\frac{d}{d\theta})\exp(\theta S))(0)+\cdots+B_mY(b_m(\frac{d}{d\theta})\exp(\theta S))(0)=\\
B_1Yb_1(S)+\cdots+B_mYb_m(S).
\end{multline*}
Moreover, by using the relation between $b_i$ and $f_i$ and $M_i$ and
$B_i$ given by \eqref{eq:BiMi} we have, 
\begin{multline}
B_1Yb_1(S)+\cdots+B_mYb_m(S)=\\M(0)^{-1}\big[M_1Y(f_1(0)I-f_1(S))S^{-1}+\cdots+M_mY(f_m(0)I-f_m(S))S^{-1}\big]=\\
M(0)^{-1}\big[\MM(Y,0)S^{-1}-\MM(Y,S)S^{-1}\big]=
YS^{-1}-M(0)^{-1}\MM(Y,S)S^{-1}.\label{eq:BYbk}
\end{multline}
The proof is completed by cancelling the term $YS^{-1}$ when inserting \eqref{eq:BYbk} into \eqref{eq:BBBF0}.
\end{proof}

\end{document}